\documentclass[a4paper, 10pt]{amsart}
\input cyracc.def

\usepackage{amsmath}
\usepackage{amssymb}
\usepackage{amsfonts}
\usepackage{eucal}
\usepackage{url}
\usepackage{bm, comment}
\usepackage{hyperref}
\usepackage{mathrsfs}
\usepackage{graphicx}
\usepackage{amscd}
\usepackage{amsthm}
\usepackage{mathtools}
\usepackage{stmaryrd}
\usepackage{MnSymbol}
\usepackage[all]{xy}
\usepackage{tikz}
\usetikzlibrary{cd}
\usetikzlibrary{patterns}

\DeclareMathOperator{\clo}{clo}

\DeclareMathOperator{\diag}{diag}

\DeclareMathOperator{\fun}{fun}

\DeclareMathOperator{\id}{id}

\DeclareMathOperator{\op}{op}

\DeclareMathOperator{\ps}{ps}
\DeclareMathOperator{\Proj}{Proj}

\DeclareMathOperator{\red}{red}

\DeclareMathOperator{\std}{std}
\DeclareMathOperator{\Spec}{Spec}

\DeclareMathOperator{\CAlg}{CAlg}

\DeclareMathOperator{\GL}{GL}

\DeclareMathOperator{\SL}{SL}
\DeclareMathOperator{\SO}{SO}

\DeclareMathOperator{\Oo}{O}

\newcommand{\bC}{\mathbb{C}}

\newcommand{\bG}{\mathbb{G}}

\newcommand{\bP}{\mathbb{P}}

\newcommand{\bR}{\mathbb{R}}
\newcommand{\bZ}{\mathbb{Z}}

\newcommand{\cB}{\mathcal{B}}

\newcommand{\cF}{\mathcal{F}}
\newcommand{\cG}{\mathcal{G}}

\newcommand{\cL}{\mathcal{L}}

\newcommand{\cO}{\mathcal{O}}

\newcommand{\cV}{\mathcal{V}}

\newcommand{\fa}{\mathfrak{a}}
\newcommand{\fb}{\mathfrak{b}}
\newcommand{\fg}{\mathfrak{g}}
\newcommand{\fgl}{\mathfrak{gl}}
\newcommand{\fh}{\mathfrak{h}}

\newcommand{\fm}{\mathfrak{m}}

\newcommand{\fp}{\mathfrak{p}}
\newcommand{\fq}{\mathfrak{q}}
\newcommand{\fsl}{\mathfrak{sl}}
\newcommand{\fso}{\mathfrak{so}}
\newcommand{\ft}{\mathfrak{t}}

\theoremstyle{plain}
\newtheorem{thm}{Theorem}[section]
\newtheorem{cor}[thm]{Corollary}
\newtheorem{lem}[thm]{Lemma}
\newtheorem{prop}[thm]{Proposition}

\theoremstyle{definition}

\newtheorem{ex}[thm]{Example}

\newtheorem{prob}[thm]{Problem}
\newtheorem{property}[thm]{Property}
\newtheorem{rem}[thm]{Remark}
\begin{document}
	\title{$\SO(3)$-homogeneous decomposition of the flag scheme of $\SL_3$ over $\bZ\left[1/2\right]$}
	\author{Takuma Hayashi}
	\address{Department of Pure and Applied Mathematics, Graduate School of Information Science and Technology, Osaka University, 1-5 Yamadaoka, Suita, Osaka 565-0871, Japan}
	\email{hayashi-t@ist.osaka-u.ac.jp}
	\date{}
	\maketitle
	\begin{abstract}
		In this paper, we give $\bZ\left[1/2\right]$-forms of $\SO(3,\bC)$-orbits in the flag variety of $\SL_3(\bC)$. We also prove that they give a $\bZ\left[1/2\right]$-form of the $\SO(3,\bC)$-orbit decomposition of the flag variety of $\SL_3$.
	\end{abstract}
	\section{Introduction}\label{sec:intro}
	
	Motivated by applications to special values of automorphic $L$-functions, Michael Harris, G\"unter Harder, and Fabian Januszewski started to work on $(\fg,K)$-modules over number fields and localization of the rings of their integers in the 2010s (\cite{MR3053412,MR4073199, MR3970997, MR3770183, MR3937337, 1604.04253}). For general theory of $(\fg,K)$-modules over commutative rings, see \cite{MR4007195,MR3853058}. Among those, Harris proposed to construct rational models of discrete series representations from the corresponding closed $K$-orbits in the flag variety and line bundles on them over the field $\bC$ of complex numbers through the localization. In \cite{hayashijanuszewski}, we studied descent properties of rings of definition of certain closed $K$-orbits in the moduli scheme of parabolic subgroups of $G$ for reductive group schemes $K\subset G$ in the sense of \cite[D\'efinition 2.7]{MR0228502}. As a consequence, we established real and smaller arithmetic forms of $A_{\fq}(\lambda)$-modules (\cite[Section 6.2]{hayashijanuszewski}).
	
	In this paper, we study rings of definition of the remaining three $\SO(3,\bC)$-orbits in the complex flag variety of $\SL_3$. The main result is to establish a $\bZ\left[1/2\right]$-analog of the $\SO(3,\bC)$-orbit decomposition of the complex flag variety of $\SL_3$:
	
	\begin{thm}[Theorem \ref{thm:openorbit}, Theorem \ref{thm:iisimmersion}, Lemma \ref{lem:iclo}, Theorem \ref{mainthm}]\label{roughstatement}
		The flag scheme $\cB_{\SL_3}$ of $\SL_3$ over $\bZ\left[1/2\right]$ is decomposed into four affinely imbedded subschemes which are $\SO(3)$-homogeneous in the \'etale topology in the sense of \cite[Proposition et d\'efintion 6.7.3]{MR0237513}.
	\end{thm}
	
	\subsection{First Perspective: $(\fg,K)$-Modules over Commutative Rings}\label{sec:representationtheory}
	
	In the representation theory of Lie groups, there are many phenomena which we can understand through real and complex geometry. For instance, there are two geometric realizations of principal series representations of the special linear Lie group $\SL_3(\bR)$. For simplicity, we restrict ourselves to (the Harish-Chandra module of) the principal series representation $X_{\ps}$ with trivial parameter. We can realize $X_{\ps}$ by using the real flag manifold $\SL_3(\bR)/B_{\std}(\bR)$, where $B_{\std}(\bR)$ is the Borel subgroup of $\SL_3(\bR)$ consisting of upper triangular matrices. The representation $X_{\ps}$ can be realized as the space of functions of $\SL_3(\bR)/B_{\std}(\bR)$. According to the Iwasawa decomposition, it can be identified with $\SO(3,\bR)/B_{\std}(\bR)\cap\SO(3,\bR)$, where $\SO(3,\bR)$ is the special orthogonal group. This is a geometric explanation why $X_{\ps}$ is induced from the trivial representation of $\SO(3,\bR)\cap B_{\std}(\bR)$ as a representation of $\SO(3,\bR)$. We also note that $X_{\ps}$ admits a natural real structure by this construction. The other realization is to use the complex flag variety of $\SL_3$. We define the complex algebraic groups $\SL_3(\bC)$, $B_{\std}(\bC)$, $\SO(3,\bC)$ in a similar way. Then, we have a unique open $\SO(3,\bC)$-orbit in the complex flag variety $\SL_3(\bC)/B_{\std}(\bC)$. As a complex $\SO(3,\bC)$-variety, it is given by $\SO(3,\bC)/\SO(3,\bC)\cap B_{\std}(\bC)$. The representation $X_{\ps}$ can be realized as the space of regular functions on this orbit. We can think of these regular functions as global sections of the pushforward of the coordinate ring of this orbit to $\SL_3(\bC)/B_{\std}(\bC)$. These two realizations are related by the analytic continuation.
	
	What happens to other kinds of representations? The Beilinson-Bernstein correspondence tells us that irreducible Harish-Chandra modules of $\SL_3(\bR)$ with trivial infinitesimal character are obtained by D-modules on the complex flag variety. For example, the fundamental representation $A_\fb(0)$ is attached to the unique closed $\SO(3,\bC)$-orbit. Can we obtain them from real geometric objects? The transitivity of the $\SO(3)$-action on the real flag manifold of $\SL_3$ tells us that principal series representations should be the only representations which we can obtain from the real flag \textit{manifold}. 
	
	To give another nice answer to the above question, we shall change the point of view. Since there is only one closed $\SO(3,\bC)$-orbit in $\SL_3(\bC)/B_{\std}(\bC)$, this orbit is stable under formation of the complex conjugation. Hence, it is naturally defined over the real numbers by Galois descent. In fact, we realized this object as the real flag \textit{variety} of $\SO(3)$ in \cite[Section 5]{hayashijanuszewski}. More strongly, we gave an equivariant closed immersion of flag schemes over $\bZ\left[1/2\right]$. The point is that this real algebraic variety which does exist does not admit real points. Hence, it cannot appear in the formalism of manifolds. It is not an $\SO(3)$-orbit as well by the same reason. However, it is homogeneous in the sense of \cite[Proposition et d\'efintion 6.7.3]{MR0237513}. This appears to be a nice formalism to study ``orbit-like'' objects over general base fields (rings). This real algebraic subvariety gives rise to a real form of the fundamental series representation $A_\fb(0)$ via localization (\cite[Section 6.2]{hayashijanuszewski}). The existence of the real form of $A_\fb(0)$ was proved algebraically in \cite[Theorem 7.3]{MR3770183}.
	
	The main purpose of this paper is to construct $\bZ\left[1/2\right]$-forms of the remaining $\SO(3,\bC)$-orbits explicitly. Philosophically, this result says that all irreducible admissible representations of $\SL_3(\bR)$ (with regular infinitesimal character) are controlled by real algebro-geometric objects (cf.\ \cite[Definition 3.5.1, Corollary 4.2.2]{hayashijanuszewski}). We will even control representations over $\bZ[1/2]$ at the level of orbits. As a consequence of our direct computational approach, we find that the $\bZ[1/2]$-forms of the orbits are represented by affine schemes except the closed one. This fact makes the study of the global section modules of the direct and proper direct image twisted D-modules (cf.~\cite[Appendix A]{hayashifil}).

	\subsection{Second Perspective: Combinatorics of Orbit Decomposition}\label{sec:combinatorics}
	
	Traditional problems in the theory of combinatorics of orbit decomposition are summarized as follows:
	\begin{prob}\label{prob:orbit}
		Let $k$ be a commutative ring, and $X$ be a $k$-scheme, equipped with an action of a group scheme $K$ over $k$. Classify the $K(F)$-orbits in $X(F)$ for a field $F$ over $k$.
	\end{prob}
	
	The ring $k$ in Problem \ref{prob:orbit} should be usually a certain localization of the ring of integers or the field of rational numbers. The description of the classification depends on $F$ in most cases. For instance, the number of orbits may differ by $F$. On the other hand, one can find by experience that some parts of the classifications are independent of $F$. In fact, we (possibly implicitly) happen to solve equations on the course of classification by nature of the algebro-geometric formulation. Among those equations, some may be defined over $k$. As far as such equations are concerned, there are two factors why the dependence happens:
	\begin{enumerate}
		\item[1.] Existence of solutions;
		\item[2.] Dependence on choice of solutions (Galois symmetry).
	\end{enumerate}
	Hence, the larger rings $k'$ we replace $k$ by, the more uniform classification we obtain.
	
	In this paper, we suggest the three things.
	\begin{enumerate}
		\item[1.] We quit solving equations whose solutions essentially depend on fields $F$.
		\item[2.] Attach a Galois extension $k\to k'$ with Galois group $\Gamma$ to an independent equation. Then, we find $\Gamma$-invariant parts of a decomposition of $X(F')$ for fields $F'$ over $k'$.
		\item[3.] We do these things at the level of schemes to get a decomposition of $X$ into $K$-invariant subspaces $Z_\lambda$ as a set.
	\end{enumerate}
	Decomposition into subspaces $Z_\lambda$ may not give a complete answer to Problem \ref{prob:orbit} but instead that we obtain a uniform decomposition of $X(F)$ in $F$ into $K(F)$-invariant subsets $Z_\lambda(F)$. In fact, $Z_\lambda(F)$ can have multiple $K(F)$-orbits. We explain below that $Z_\lambda(F)$ can be also empty.
	
	\begin{prob}
		Decompose $X$ into smaller pieces of $K$-invariant subspaces represented by $k$-schemes.
	\end{prob}
	
	We would like to suggest a basic strategy to get $Z_\lambda$, which consists of four phases.
	\begin{enumerate}
		\item[1.] Take a Galois extension $k\to k'$.
		\item[2.] Define $K\otimes_k k'$-orbits of $X\otimes_k k'$ by taking $k'$-points of $X$.
		\item[3.] Prove that the $K\otimes_k k'$-orbits decompose $X\otimes_k k'$ as a set.
		\item[4.] Study the Galois orbit of the set of $K\otimes_k k'$-orbits to get subspaces over $k$ by the Galois descent.
	\end{enumerate}
	The last phase says that we prove that a $K\otimes_k k'$-orbit admits a $k$-form by showing that it admits a Galois action; otherwise, we get a $k$-scheme by joining $K\otimes_k k'$-orbits. For example, see \cite[Example 5.2.22]{hayashijanuszewski}. We can regard that \cite[Proposition 5.1.1]{hayashijanuszewski} is the case that the Galois orbit has two elements ($K=\SO(2)$, $k=\bZ\left[1/2\right]$, $k'=\bZ\left[1/2,\sqrt{-1}\right]$). This is the reason why $Z_\lambda(k)$ may be empty. This observation tells us that even if the Galois orbit is a singleton, the expected ``orbits'' (subspaces) $Z_\lambda$ may not have a base point. The author believes that the key ingredients to achieve each phase of our program lie in the combinatorial study of Problem \ref{prob:orbit}.
	
	More specifically, we shall think of the $K_\bC$-orbit decomposition of flag varieties. For this, recall the Matsuki classification (\cite{MR527548}, see also \cite{MR1066573} for 
	similar results over algebraically closed fields of characteristic $\neq 2$). Let $G$ be a connected real reductive algebraic group, and $\fg_0$ be its Lie algebra. Let $K$ be a maximal compact subgroup of the group $G(\bR)$ of real points of $G$, and $K_\bC$ be its complexification. Let $\theta$ be the Cartan involution relative to $K$. For a $\theta$-stable Cartan subalgebra $\fh_0\subset\fg_0$, set
	\[\begin{array}{ccc}
		\fh=\fh_0\otimes_\bR \bC,&	W_G(\fg,\fh)=N_{G(\bC)}(\fh)/Z_{G(\bC)}(\fh),
		&W_G(\fg,\fh)^\sigma=N_{K}(\fh)/Z_{K}(\fh),
	\end{array}\]
	where $N$ and $Z$ denote the normalizer and the centralizer respectively. Let $\cB_G$ be the flag variety of $G$. There exists a bijection
	\[K_\bC\backslash \cB_G(\bC)\cong\coprod_{\fh_0}W_G(\fg,\fh)^\sigma\backslash W_G(\fg,\fh),\]
	where $\fh_0$ runs through fixed representatives of $K$-conjugacy classes of $\theta$-stable Cartan subalgebras.
	\begin{ex}
		The special orthogonal group $\SO(2,\bR)$ acts on the real projective line $\bP^1(\bR)$ transitively. On the other hand, there are three $\SO(2,\bC)$-orbits in $\bP^1(\bC)$. This difference suggests us to decompose $\bP^1$ into the $\SO(2)$-orbit containing real points and the others which are closed. Since the complex conjugation switches the two closed orbits $\{\sqrt{-1}\}$ and $\{-\sqrt{-1}\}$ in $\mathbb{C}\cup\{\infty\}\cong\mathbb{P}^1(\mathbb{C})$, the two $\SO(2,\bC)$-orbits are not defined over $\bR$ but their union is. Technically, observe that the Galois action $\sqrt{-1}\mapsto -\sqrt{-1}$ is represented by the diagonal matrix $\diag(1,-1)\in\Oo(2,\bR)\setminus\SO(2,\bR)$, where $\Oo(2,\bR)$ is the orthogonal group. Moreover, $\Oo(2,\bC)$ is generated by $\SO(2,\bC)$ and this matrix. Therefore, the union is an $\Oo(2,\bC)$-orbit defined over $\bR$. Based on this idea, we proved a $\bZ$-analog of the $\Oo(2,\bC)$-orbit decomposition of the flag variety of $\GL_2$ in \cite[Section 5.1]{hayashijanuszewski}. We summarize the present observation as follows: the closed $\SO(2,\bC)$-orbits are not defined over $\bR$ but that the closed $\Oo(2,\bC)$-orbit is so. This difference comes from the fact $\diag(1,-1)\in\Oo(2,\bR)\setminus\SO(2,\bR)$. This matrix appears in the Matsuki classification. In fact, we denote split Cartan subalgebras of the real special linear Lie algebra $\fsl_2(\bR)$ and the real general linear lie algebra $\fgl_2(\bR)$ by the same symbol $\fh_{\std,0}$, and denote the fundamental Cartan subalgebras of $\fsl_2(\bR)$ and $\fgl_2(\bR)$ containing the orthogonal Lie algebra $\fso(2,\bR)$ by the same symbol $\fh_{\fun,0}$. Then, we have
		\[W_{\SL_2}(\fsl_2,\fh_{\std})^\sigma\cong W_{\SL_2}(\fsl_2,\fh_{\std})\cong W_{\GL_2}(\fgl_2,\fh_{\std})\cong W_{\GL_2}(\fgl_2,\fh_{\std})^\sigma,\]
		\[\begin{array}{cc}
			W_{\SL_2}(\fsl_2,\fh_{\fun})^\sigma=\{1\},
			&W_{\GL_2}(\fgl_2,\fh_{\fun})^\sigma\cong\bZ/2\bZ.
		\end{array}\]
		The nontrivial element of $W_{\GL_2}(\fg,\fh_{\fun})^\sigma$ is represented by $\diag(1,-1)$.
	\end{ex}
	\begin{ex}\label{ex:sl_3}
		Put $G=\SL_3$. The compact Lie group $\SO(3,\bR)$ acts transitively on $\cB_{\SL_3}(\bR)$, and there are four $\SO(3,\bC)$-orbits in $\cB_{\SL_3}(\bC)$. One should separate the $\SO(3)$-orbit containing real points and the others. We can see that the unique closed $\SO(3,\bC)$-orbit is defined over the real numbers from \cite[Proposition 5.5.2]{hayashijanuszewski}. The key idea for its proof is that a corresponding Borel subgroup of $\SL_3$ is determined by a regular cocharacter to $\SO(3,\bC)$, and that its conjugation is expressed by the action of $w_0=\diag(1,-1,-1)\in\SO(3,\bR)$.
		That is, $w_0$ plays the role of the Galois action. This should happen since $w_0=-1$ as an element of the Weyl group of $\SO(3,\bC)$. The main idea of this paper is as follows: $w_0$ also plays the role of the Galois action to certain Borel subgroups corresponding to the other two orbits since the Weyl group of $\SL_3(\bC)$ does not contain $-1$. For example, we can explain it at the Lie algebra level as follows: set
		\[\ft_{\fun,0}=\bR\left(\begin{array}{ccc}
			0&1&0\\
			-1&0&0\\
			0&0&0
		\end{array}\right),\]
		and $\fh_{\fun,0}\subset\fsl(3,\bR)$ be the fundamental Cartan subalgebra containing $\ft_{\fun,0}$. Let
		$\fh_{\fun,0}=\ft_{\fun,0}\oplus \fa_{\fun,0}$
		denote the Cartan decomposition. Then, $w_0$ acts on $\ft_{\fun,0}$ by $-1$. Since $w_0\neq -1$ and $w_0^2=1$ as an element of the Weyl group of $\SL_3$, and $\dim\fa_{\fun,0}=1$, $w_0$ acts on $\fa_{\fun,0}$ by 1. Since an element of $\fh_{\fun}$ determining each of the above two Borel subgroups belongs to $\sqrt{-1}\ft_{\fun,0}\oplus\fa_{\fun,0}$, the action of $w_0$ on this element coincides with the conjugate action. In this paper, we improve this idea to work over $\bZ\left[1/2\right]$.
	\end{ex}
	
	Like the last part of Example \ref{ex:sl_3}, we will have to analyze combinatorial results carefully in general to get hints from them. The author is working in progress on a $\bZ\left[1/2\right]$-analog of the $K_\bC$-orbit decomposition of the flag varieties for higher rank classical groups by proactive use of \cite{MR1066573}. There is also a more general formalism:
	
	\begin{prob}
		Let $G$ be a reductive group scheme over $k$, and $K$ be a closed subgroup scheme of $G$. Decompose the moduli scheme of parabolic subgroups of $G$ into smaller pieces of $K$-invariant subschemes.
	\end{prob}
	
	Note that symmetric subgroups in the sense of \cite[Example 3.1.2]{MR4627704} will be typical examples of $K$.
	
	\subsection{Organization of this Paper}\label{sec:organization}
	
	In Section \ref{sec:settheoreticdecomposition}, we collect some general results on decompositions of schemes to verify ideas of Section \ref{sec:combinatorics}. In Appendix \ref{sec:descent}, we collect some general results on descent techniques in abstract algebraic geometry. They will be helpful when we try to find forms of orbit decompositions of schemes, based on our general program of Section \ref{sec:combinatorics}. In Section \ref{sec:orbit}, we use them to construct $\bZ\left[1/2\right]$-forms of the $\SO(3,\bC)$-orbits in $\cB_{\SL_3}(\bC)$. We also give their moduli descriptions. We cost many pages to this section for confirming the moduli descriptions (particularly Theorem \ref{thm:openorbit} and Lemma \ref{lem:translation1}) because we perform the Gram-Schmidt process and its versions explicitly and independently to the general matrices of size $3\times 3$. As a result, we find explicit formulas of the defining relations of the $\SO(3)$-homogeneous subschemes of $\cB_{\SL_3}$. We also cost pages to the proof of the isomorphism $\SO(3)/\SO(2)\cong \Spec\bZ[1/2,\sqrt{-1},x,y,z]/(x^2+y^2+z^2-1)$ in Proposition \ref{prop:SO(3)/SO(2)=S^2}, based on the sheaf-theoretic definition of $\SO(3)/\SO(2)$. In fact, we construct the inverse of the canonical map from left to right \'etale locally. On these courses and the formulations to these results, we meet many matrices of size $3\times 3$. Section \ref{sec:kgb} is devoted to the conclusion. In Appendix \ref{sec:flagschofso(3)}, we use ideas of this paper to give a reasonable realization of the flag scheme of $\SO(3)$ over $\bZ\left[1/2\right]$. We also establish a $\bZ\left[1/2\right]$-form of the $\SO(3,\bC)$-orbit decomposition of a proper complex partial flag variety of $\SL_3$. We again meet matrices of size $3\times 3$ which take large space. Totally, the many pages are needed for the case-by-case studies of the orbits through computations of large matrices.
	
	\subsection{Notation}\label{sec:notation}
	
	We follow \cite{MR4627704} for the notations and conventions. In the below, we list additional notations.
	
	To save space, we denote vertical vectors in $R^3$ as $(a_1~a_2~a_3)^T$ for a commutative ring $R$.
	
	For a field $F$, we denote its algebraic closure by $\bar{F}$.
	
	Let $k$ be a commutative ring. Let $\CAlg_k^{\red}$ denote the full subcategory of $\CAlg_k$ consisting of reduced $k$-algebras. A sheaf on the big affine \'etale site over $k$ will be called an \'etale $k$-sheaf, which will be identified with a copresheaf on $\CAlg_k$ in this paper. If necessary, see \cite[Section 2]{MR3495343} for the general formalism of sheaves on sites. We will regard $k$-schemes as \'etale $k$-sheaves by the restricted Yoneda functor (see \cite[Theorem 4.1.2]{MR3495343}). For a copresheaf $\cF$ on $\CAlg_k$ and a $k$-algebra $R$, we will sometimes identify an element $x\in\cF(R)$ with a natural transformation $\Spec R\to \cF$ of copresheaves by the Yoneda lemma. It is evident by definitions that for any homomorphism $k\to k'$ of commutative rings, the base change $-\otimes_k k'$ sends \'etale $k$-sheaves to \'etale $k'$-sheaves.
	
	For the formalism of quotient by group schemes, we adopt the quotient in the \'etale topology. That is, let $k$ be a commutative ring, $G$ be a group $k$-scheme, and $H\subset G$ be a subsegroup $k$-scheme. Then, $G/H$ is the \'etale sheafification of the $k$-space defined by $R\mapsto G(R)/H(R)$. See \cite{MR0237513} for the general formalism. Although $G/H$ is not represented by a $k$-scheme in general, we will see that the quotients appearing in this paper are representable.
	
	For a reductive group scheme $G$ over a scheme $S$ in the sense of \cite[D\'efinition 2.7]{MR0228502}, the moduli scheme of Borel subgroups of $G$ will be denoted by $\mathcal B_G$ (see \cite[Corollaire 5.8.3 (i)]{MR0218363}).
	
	Let $X$ be a scheme over $k$. Let $|X|$ denote the underlying set of $X$. We will use a similar notation for morphisms of schemes. For a point $x\in |X|$, let $\kappa(x)$ denote the residue field at $x$. If we are given a field $F$ over $k$ and an element $x\in X(F)$, we will also denote the residue field of $X$ at the image of the point of $|\Spec F|$ along the map $|x|:|\Spec F|\to |X|$ by the same symbol $\kappa(x)$. For a point $x\in |X|$, we denote the geometric point $\Spec\overline{\kappa(x)}\to\Spec\kappa(x)\to X$ by $\bar{x}$.
	
	Let $B_{\std}$ denote the Borel subgroup of $\SL_3$ over $\bZ\left[1/2\right]$ consisting of upper triangular matrices.
	
	\section{Remark on Set-Theoretic Decomposition of Schemes}\label{sec:settheoreticdecomposition}
	
	Fix $k$ as a commutative ground ring. Let $X$ be a $k$-scheme, and $\{i_\lambda:Z_\lambda\to X\}_{\lambda\in\Lambda}$ be a small set of monomorphisms of $k$-schemes. We say that $\{Z_\lambda\}$ exhibits a set-theoretic decomposition of $X$ if the canonical map
	$\coprod_{\lambda\in\Lambda} |Z_\lambda|\to |X|$
	is a bijection. The goal of this paper is to decompose the flag scheme $\cB_{\SL_3}$ into $\SO(3)$-invariant subschemes as a set. In this section, we note some general results on set-theoretic decompositions of schemes to relate them with the results of combinatorics.
	
	\begin{thm}\label{thm:fieldwisedec}
		The following conditions are equivalent:
		\begin{enumerate}
			\renewcommand{\labelenumi}{(\alph{enumi})}
			\item The canonical map
			$\coprod_{\lambda\in\Lambda} |Z_\lambda|\to |X|$
			is a bijection.
			\item The canonical map
			$\coprod_{\lambda\in\Lambda} Z_\lambda(F)\to X(F)$
			be a bijection for every field $F$ over $k$.
			\item The canonical map
			$\coprod_{\lambda\in\Lambda} Z_\lambda(F)\to X(F)$
			be a bijection for every algebraically closed field $F$ over $k$.
		\end{enumerate}
	\end{thm}
	\begin{proof}	
		It is clear that (b) implies (c). Suppose that $\{Z_\lambda\}$ satisfies (c). We prove $\{Z_\lambda\}$ exhibits a set-theoretic decomposition of $X$. Let $x\in |X|$. Since the map
		\[(i_\lambda):\coprod_\lambda Z_\lambda(\overline{\kappa(x)})\to X(\overline{\kappa(x)})\]
		is a bijection, there exist an index $\lambda$ and a unique element $z_\lambda\in Z_\lambda(\overline{\kappa(x)})$ such that $i_\lambda(z_\lambda)=\bar{x}$. The point $x$ is the image of the unique point of $|\Spec \kappa(x)|$ along the composite map $|\bar{x}|=|i_\lambda|\circ|z_\lambda|:\mid\Spec \overline{\kappa(x)}\mid\to \mid X\mid$.
		This shows that the map $\coprod_\lambda |Z_\lambda|\to |X|$ is surjective. To see that it is injective, recall that each map $|i_\lambda|:|Z_\lambda|\to |X|$ is injective by \cite[Remarque 8.11.5.1]{MR0217086}. Hence it will suffice to show that the images of $|i_\lambda|$ are disjoint in $|X|$. Let $\lambda,\mu\in\Lambda$ be distinct indices, $z_\lambda\in |Z_\lambda|$, and $z_\mu\in|Z_{\mu}|$. Suppose that $|i_\lambda|(z_\lambda)=|i_{\mu}|(z_{\mu})=:x$. Then, we canonically obtain a commutative diagram
		\[\begin{tikzcd}
			\Spec\kappa(z_\lambda)\ar[r, "z_\lambda"]\ar[rrd]&Z_\lambda\ar[r, "i_\lambda"]&X
			&Z_{\mu}\ar[l, "i_{\mu}"']&\Spec \kappa(z_{\lambda'})\ar[l, "z_{\lambda'}"']\ar[lld]\\
			&&\Spec\kappa(x)\ar[u, "x"]
		\end{tikzcd}\]
		Then, we can find an algebraically closed field $F$ enjoying a commutative diagram
		\[\begin{tikzcd}
			\Spec\kappa(z_\lambda)\ar[r]&\Spec\kappa(x)&\Spec \kappa(z_{\mu})\ar[l]\\
			&\Spec F.\ar[u]\ar[ru]\ar[lu]
		\end{tikzcd}\]
		This shows that $x|_{\Spec F}$ is in the images of both $i_\lambda$ and $i_{\mu}$. It contradicts to the assumption that $\{Z_\lambda\}$ exhibits a fieldwise decomposition. This proves the implication (c)$\Rightarrow$(a).
		
		Finally, suppose that $\{Z_\lambda\}$ satisfies (a). We wish to show that $\{Z_\lambda\}$ satisfies (b). Fix a field $F$ over $k$. To see that the map
		$(i_\lambda):\coprod_{\lambda\in\Lambda}Z_\lambda(F)\to X(F)$
		is injective, it will suffice to show that the images of $Z_\lambda(F)$ in $X(F)$ are disjoint since $i_\lambda$ are monomorphisms. Let $\lambda,\mu\in\Lambda$ be distinct indices, $z_\lambda\in Z_\lambda(F)$, and $z_{\mu}\in Z_{\mu}(F)$. Let us wirte the corresponding elements in $|Z_\lambda|$ and $|Z_{\mu}|$ by the same symbols $z_\lambda$ and $z_{\mu}$ respectively. Suppose $i_\lambda(z_\lambda)=i_{\mu}(z_{\mu})$. Then, we have $|i_\lambda|(z_\lambda)=|i_{\mu}|(z_{\mu})$, which contradicts to the condition (a) since $\lambda\neq\mu$.
		
		The proof is completed by showing that
		$(i_\lambda):\coprod_{\lambda\in\Lambda}Z_\lambda(F)\to X(F)$
		is surjective. Let $x\in X(F)$. Then, the corresponding element $x\in |X|$ can be expressed as $|i_\lambda|(z_\lambda)$ for some index $\lambda$ and an element $z_\lambda\in |Z_\lambda|$. Let $Z_{\lambda,x}$ be the fiber of $i_\lambda$ at $x\in |X|$. Consider the commutative diagram
		\[\begin{tikzcd}
			&Z_{\lambda,x}\ar[r]\ar[d,"\sim"{sloped,below}]&Z_\lambda\ar[d, "i_\lambda"]\\
			\Spec F\ar[r]\ar[rr, bend right, "x"']&\Spec\kappa(x)\ar[r]&X.
		\end{tikzcd}\]
		The left vertical arrow in this diagram is an isomorphism by \cite[Remarque 8.11.5.1]{MR0217086} since $Z_{\lambda,x}$ is nonempty. Hence the morphism $i_\lambda$ sends the element of $Z_\lambda(F)$ given by $\Spec F\to\Spec\kappa(x)\cong Z_{\lambda,x}\to Z_\lambda$
		to $x$.
	\end{proof}
	\begin{cor}\label{cor:basechangeofdecomposition}
		Let $k'$ be a $k$-algebra. If $\{Z_\lambda\}$ exhibits a set-theoretic decomposition of $X$, $\{Z_\lambda\otimes_k k'\}$ exhibits a set-theoretic decomposition of $X\otimes_k k'$. The converse holds if $k'$ is a faithfully flat $k$-algebra.
	\end{cor}
	\begin{proof}
		The first part is clear from the definition of base changes in terms of copresheaves. Let $k'$ be a faithfully flat $k$-algebra. Let $F$ be an algebraically closed field over $k$. Then, $k'\otimes_k F$ is nonzero by the hypothesis on $k'$. Choose an algebraically closed field $F'$ over $k'\otimes_k F$. For example, it is given by $\overline{(k'\otimes_k F)/\fm}$ for some maximal ideal $\fm$ of $k'\otimes_k F$ which exists since $k'\otimes_k F\neq 0$. In particular, $F'$ is an algebraically closed field over $k'$ via the canonical homomorphism $k'\to k'\otimes_k F$. Since the embedding $F\to F'$ is a $k$-algebra homomorphism, we get a commutative diagram
		\[\begin{tikzcd}
			\coprod_{\lambda\in\Lambda} Z_\lambda(F)\ar[r, "(i_\lambda)"]\ar[d]& X(F)\ar[d]\\
			\coprod_{\lambda\in\Lambda} Z_\lambda(F')\ar[r, "(i_\lambda)"]\ar[d, equal]& X(F')\ar[d, equal]\\
			\coprod_{\lambda\in\Lambda} (Z_\lambda\otimes_k k')(F')\ar[r, "\sim" sloped]& (X\otimes_k k')(F'),
		\end{tikzcd}\]
		where the bottom arrow is a bijection since $\{Z_\lambda\otimes_k k'\}$ exhibits a fieldwise decomposition of $X\otimes_k k'$. The left upper vertical arrow is injective since the embedding $F\to F'$ is faithfully flat. Therefore, the upper horizontal arrow is injective.
		
		Suppose that we are given an element $x\in X(F)$. Then, there exist an index $\lambda\in\Lambda$ and $z_\lambda'\in Z_\lambda(F')$ such that $x|_{\Spec F'}=i_\lambda(z'_\lambda)$. Consider the canonical map $\iota_j:F'\to F'\otimes_F F'$ to the $j$th factor ($j\in\{1,2\}$).
		Since $x\in X(F)$, we have
		\[\begin{array}{cccc}
			i_\lambda(Z_\lambda(\iota_1)z_\lambda')
			&=X(\iota_1)(i_\lambda(z'_\lambda))
			&=X(\iota_1)(x|_{\Spec F'})
			&=X(\iota_2)(x|_{\Spec F'})\\
			&=X(\iota_2)(i_\lambda(z'_\lambda))
			&=i_\lambda(Z_\lambda(\iota_2)z_\lambda').
		\end{array}\]
		Since $i_\lambda$ is a monomorphism, this implies $Z_\lambda(\iota_1)z_\lambda'=Z_\lambda(\iota_2)z_\lambda'$. Since $Z_\lambda$ is a sheaf in the fpqc topology, there is a unique element $z_\lambda\in Z_\lambda(F)$ such that $z_\lambda|_{\Spec F'}=z_\lambda'$. Since the restriction $X(F)\to X(F')$ is injective by the faithfully flat descent, the equality $i_\lambda(z_\lambda)=x$ follows from
		$i_\lambda(z_\lambda)|_{\Spec F'}=i_\lambda(z_\lambda|_{\Spec F'})=i_\lambda(z_\lambda')=x|_{\Spec F'}$.
		This completes the proof.
	\end{proof}
	
	Let $K$ be a group scheme over $k$, $X$ be a scheme over $k$, equipped with an action of $K$. Let $\{i_\lambda:Z_\lambda\to X\}$ be a set-theoretic decomposition of $X$. Suppose that for each index $\lambda$, the following conditions are satisfied:
	\begin{enumerate}
		\renewcommand{\labelenumi}{(\roman{enumi})}
		\item The action of $K$ on $X$ restricts to $Z_\lambda$.
		\item Every geometric fiber of $Z_\lambda$ is nonempty and locally of finite type.
		\item For every algebraically closed field $F$ over $k$, $K(F)$ acts transitively on $Z_\lambda(F)$.
	\end{enumerate}
	Such a set-theoretic decomposition is minimal in the following sense:
	\begin{cor}
		For each $\lambda$, suppose that we are given a set-theoretic decomposition $\{Z'_{\lambda\mu}\hookrightarrow Z_\lambda\}$. Then, each set $\{Z'_{\lambda\mu}\hookrightarrow Z_\lambda\}$ is a singleton if the following conditions are satisfied for every pair $(\lambda,\mu)$:
		\begin{enumerate}
			\renewcommand{\labelenumi}{(\roman{enumi})}
			\item The action of $K$ on $X$ restricts to $Z'_{\lambda\mu}$.
			\item Every geometric fiber of $Z'_{\lambda\mu}$ is nonempty and locally of finite type.
		\end{enumerate}
	\end{cor}
	\begin{proof}
		This is an immediate consequence of Hilbert's Nullstellensatz.
	\end{proof}
	\section{Construction of $\bZ\left[1/2\right]$-Forms of $\SO(3,\bC)$-Orbits}\label{sec:orbit}
	We constructed a $\bZ\left[1/2\right]$-form of the closed $\SO(3,\bC)$-orbit in $\cB_{\SL_3}(\bC)$ in \cite{hayashijanuszewski} by the Galois descent. In this section, we construct $\bZ\left[1/2\right]$-forms of the remaining three $\SO(3,\bC)$-orbits. We also give their moduli descriptions. To achieve them, remark that if we are given a Borel subgroup of $\SL_3$ over a $\bZ\left[1/2\right]$-algebra $R$, the stabilizer of the action of $\SO(3)$ at $B\in\cB_{\SL_3}(R)$ is $\SO(3)\cap B$ since the normalizer of $B$ coincides with itself (\cite[Corollaire 5.3.12 and Proposition 5.1.3]{MR0218363}).
	
	\subsection{Preliminary Computation}\label{sec:preliminary}
	Let $R$ be a commutative ring. Let $\{e_1,e_2,e_3\}$ denote the standard basis of $R^3$. We define an $R$-bilinear form $(-,-):R^3\otimes_R R^3\to R$ by
	$(\sum_{i=1}^3 a_i e_i,\sum_{i=1}^3 b_i e_i)=\sum_{i=1}^3 a_i b_i$.
	
	Let $g\in\SL_3(R)$. For $i\in \{1,2,3\}$, write
	$v_i(g)=(g_{1i}~g_{2i}~g_{3i})^T$.
	Set
	\[c_1(g)=(v_1(g),v_1(g))\]
	\[c_2(g)=(v_1(g),v_1(g))(v_2(g),v_2(g))-(v_1(g),v_2(g))^2.\]
	\[c_3(g)=(v_1(g),v_2(g)).\]
	We will omit $(g)$ if the matrix $g$ is clear from the context.
	\begin{lem}\label{lem:kgbreduction}
		Let $g\in\SL_3(R)$, $b\in B_{\std}(R)$, and $k\in\SO(3,R)$.
		\begin{enumerate}
			\renewcommand{\labelenumi}{(\arabic{enumi})}
			\item For $i\in \{1,2\}$, $c_i(g)\in R^\times$ (resp.\ $c_i(g)=0$) if and only if $c_i(kgb)\in R^\times$ (resp.\ $c_i(kgb)=0$).
			\item Suppose that $c_1(g)=0$. We then have $c_3(g)\in R^\times$ (resp.\ $c_3(g)=0$) if and only if $c_3(kgb)\in R^\times$ (resp.\ $c_3(kgb)=0$).
		\end{enumerate}
	\end{lem}
	\begin{proof}
		We remark that $kv_i(g)=v_i(kg)$ for $i\in\{1,2,3\}$. Since $k$ respects the bilinear form $(-,-)$ on $R^3$, we have $c_1(kgb)=c_1(gb)$, $c_2(kgb)=c_2(gb)$, and $c_3(kgb)=c_3(gb)$. By definitions, we have
		\[\begin{array}{cc}
			v_1(gb)=b_{11}v_1(g),
			&v_2(gb)=b_{12}v_1(g)+b_{22}v_2(g),\\
			c_1(gb)=b_{11}^2c_1(g),
			&c_2(gb)=b_{11}^2b_{22}^2c_2(g).
		\end{array}\]
		If $c_1(g)=0$, we also have $c_3(gb)=b_{11}b_{22}c_3(g)$.
		We remark that $b_{11},b_{22}\in R^\times$ since $b$ belongs to $B_{\std}(R)$. The equivalences are now obvious.
	\end{proof}
	\subsection{Open Orbit $U$}\label{sec:oporbit}
	Let $U=\SO(3)/\SO(3)\cap B_{\std}$ be the $\SO(3)$-orbit sheaf over $\bZ\left[1/2\right]$ attached to $B_{\std}\in\cB_{\SL_3}(\bZ\left[1/2\right])$, and $i_{\op}:U\hookrightarrow \cB_{\SL_3}$ be the corresponding embedding.
	\begin{lem}\label{lem:openorbit1}
		The group scheme $\SO(3)\cap B_{\std}$ is a finite \'etale diagonalizable group scheme.
	\end{lem}
	\begin{proof}
		Consider the embedding $\Spec\bZ\left[1/2,b_1,b_2\right]/(b_1^2-1,b_2^2-1)\hookrightarrow \SL_3$ over $\bZ\left[1/2\right]$ defined by $(b_1,b_2)\mapsto\diag(b_1,b_2,b_1b_2)$.
		It clearly factors through $B_{\std}\cap \SO(3)$. We prove that it is an isomorphism onto $B_{\std}\cap \SO(3)$. Let $R$ be an arbitrary $\bZ\left[1/2\right]$-algebra, and $b=(b_{ij})\in B_{\std}(R)\cap\SO(3,R)$.
		Since $b\in B_{\std}(R)\subset\SL_3(R)$, we have $\det b=b_{11}b_{22}b_{33}=1$. In particular, the diagonal entries of $b$ are units of $R$. For $i\in\{2,3\}$ (resp.\ $i\in\{1,2\}$), we have $b_{1i}=0$ (resp.\ $b_{i3}=0$) since the $(i,1)$-entry of $b^Tb$ (resp.\ the $(i,3)$-entry of $bb^T$) is $b_{11}b_{1i}$ (resp.\ $b_{i3}b_{33}$). Therefore, $b$ is diagonal. Since $b^Tb=1$, we have
		$b_{11}^2=b_{22}^2=b_{33}^2=1$.
		Since $\det b=1$, we have $b_{33}=b_{11}b_{22}$. Hence $b$ is the image of $(b_{11},b_{22})$. This completes the proof.
	\end{proof}
	We next show that $U$ is the locus in $\cB_{\SL_3}$ where we can apply the Gram-Schmidt process in order to prove that $i_{\op}$ is an affine open immersion. In other words, $U$ can be identified with the moduli scheme of flags where the Gram-Schmidt process works.
	\begin{property}
		Let $F$ be an algebraically closed field of characteristic different from 2. We say that a full flag $V=(0\subset Fv_1\subset V_2\subset F^3)$ satisfies Property (O) if $(v_1,v_1)\neq 0$, and every nonzero element $v\in V_2$ satisfies either $(v_1,v)\neq 0$ or $(v,v)\neq 0$. We remark that a nonzero vector $v\in V$ such that $(v_1,v)=0$ is unique up to nonzero scalar since $\dim V_2=2$.
	\end{property}
	\begin{thm}\label{thm:openorbit}
		\begin{enumerate}
			\renewcommand{\labelenumi}{(\arabic{enumi})}
			\item Let $R$ be an arbitrary $\bZ\left[1/2\right]$-algebra. For a Borel subgroup $B\in\cB_{\SL_3}(R)$, the following conditions are equivalent:
			\begin{enumerate}
				\item[(a)] $B$ belongs to the image of $i_{\op}$.
				\item[(b)] The flags corresponding to all geometric fibers of $B$ satisfy Property (O).
			\end{enumerate}
			\item The sheaf $U$ is represented by an affine $\bZ\left[1/2\right]$-scheme.
			\item The morphism $i_{\op}$ is an affine open immersion.
		\end{enumerate}
	\end{thm}
	\begin{proof}
		Part (2) follows from (1). In fact, $U$ is a sheaf in the fpqc topology by (1) since the condition (b) is local in the fpqc topology (use Lemma \ref{lem:ffreduction} if necessary). In particular, $U$ is the fpqc quotient of $\SO(3)$ by $\SO(3)\cap B_{\std}$. Part (2) then follows from Lemma \ref{lem:openorbit1} and \cite[Corollaire 5.6]{MR0212024} (or \cite[I.5.6 (6)]{MR2015057}).
		
		Let $R$ be a $\bZ\left[1/2\right]$-algebra, and $B\in\cB_{\SL_3}(R)$. Suppose that $B$ satisfies (a). Since (b) is local in the \'etale topology, we may pass to an \'etale cover to assume that there exists an element $k\in\SO(3,R)$ such that $B=kB_{\std}k^{-1}$ by \cite[I 5.4 (4), 5.5, 5.6 (2)]{MR2015057}. Since $B_{\std}$ satisfies (b), $B$ does so by $k\in\SO(3,R)$. Conversely, suppose that $B$ satisfies (b). Since (a) is local in the \'etale topology by Lemma \ref{lem:imisetalelocal}, we may assume that $B=gB_{\std}g^{-1}$ for some matrix $g\in\SL_3(R)$ by \cite[Proposition 5.1.3 and Corollaire 5.3.12]{MR0218363}.	Since $c_1=c_1(g)$ is nonzero at every (geometric) fiber, $c_1$ is a unit of $R$. Let us pass to the \'etale cover $\Spec R\left[\sqrt{c}_1\right]\to\Spec R$. Since $B=gbB_{\std}b^{-1}g^{-1}$ for every element $b\in B_{\std}(R)$, one can replace $g$ by
		\[g\left(\begin{array}{ccc}
			\frac{1}{\sqrt{c_1}}&-\frac{(v_1,v_2)}{c_1}&0\\
			0&1&0\\
			0&0&\sqrt{c_1}
		\end{array}\right)\]
		to assume that $(v_1,v_1)=1$ and $(v_1,v_2)=0$. Since $v_2$ is nonzero at each geometric fiber, (b) implies $c_2=(v_2,v_2)\in R^\times$. Pass to the \'etale cover $\Spec R\left[\sqrt{c}_2\right]$. Then, replace $g$ by $g\diag(1,\frac{1}{\sqrt{c_2}},\sqrt{c_2})$
		to assume $(v_2,v_2)=1$. Note that $(v_1,v_1)=1$ and $(v_1,v_2)=0$ still hold. Define a matrix $k$ by
		\[\begin{array}{cc}
			v_1(k)=v_1(g),&v_2(k)=v_2(g),
		\end{array}\]
		\[v_3(k)=v_3(g)-(v_1(g),v_3(g))v_1(g)-(v_2(g),v_3(g))v_2(g).\]
		Then, $v_1(k)$, $v_2(k)$, and $v_3(k)$ are orthogonal to each other. Since $\det$ is alternating multilinear, we have $\det k=\det g=1$. Therefore, we obtain $(v_3(k),v_3(k))=1$ from
		\[1=(\det k)^2=\det k^T k=\det ((v_i(k),v_j(k)))=(v_3(k),v_3(k)).\]
		As a consequence, $k$ belongs to $\SO(3,R)$. Since
		\[g=k
		\left(\begin{array}{ccc}
			1&0&(v_3(g),v_1(g))\\
			0&1&(v_3(g),v_2(g))\\
			0&0&1
		\end{array}\right),\]
		$B=kB_{\std}k^{-1}$ belongs to the image of $i_{\op}$. This shows (1).
		
		For (3), we show that for every test affine scheme $\Spec R$ over $\bZ\left[1/2\right]$ and a morphism $\Spec R\to \cB_{\SL_3}$, the base change $\Spec R\times_{\cB_{\SL_3}}U\to \Spec R$ is an affine open immersion. Let $B$ be the Borel subgroup corresponding to $\Spec R\to \cB_{\SL_3}$. Since the assertion is \'etale local in $\Spec R$, we may again assume $B=gB_{\std}g^{-1}$ for some $g\in\SL_3(R)$. For a ring homomorphism $f:R\to S$, the following conditions are equivalent:
		\begin{enumerate}
			\renewcommand{\labelenumi}{(\alph{enumi})}
			\item $f$ belongs to $(U\times_{\cB_{\SL_3}}\Spec R)(S)\subset(\Spec R)(S)$;
			\item $f(c_1)$ and $f(c_2)$ are nonzero in each residue field of $\Spec S$;
			\item $f(c_1)$ and $f(c_2)$ are units of $S$.
			\item The homomorphism $f$ descends to a map $R_{c_1c_2}\to S$.
		\end{enumerate}
		Therefore, $U\times_{\cB_{\SL_3}}\Spec R$ is isomorphic to $\Spec R_{c_1c_2}$. This completes the proof.
	\end{proof}
	
	\begin{rem}\label{rem:referee}
		The formula $c_1c_2$ appear more directly by the pull back of this open subscheme along the projection $\SL_3\to \SL_3/B_{\std}\cong\cB_{\SL_3}$. That is, the open subscheme of $\SL_3$ obtained by this base change is defined by $c_1c_2$. Similar results hold in the forms of the other orbits below.
	\end{rem}
	
	\subsection{Middle Subschemes $Z_1$ and $Z_2$}\label{sec:middleorbit}
	We next construct by Galois descent $\bZ\left[1/2\right]$-forms of the two orbits which are neither open or closed. Let $\Gamma=\bZ/2\bZ$, and $\sigma$ denote its nontrivial element. Recall that $\bZ\left[1/2\right]\subset\bZ\left[1/2,\sqrt{-1}\right]$ is a Galois extension of Galois group $\Gamma$ for the conjugation $\sqrt{-1}\mapsto -\sqrt{-1}$.
	
	Set
	\[g_1=\left(\begin{array}{ccc}
		1&-\sqrt{-1}&0\\
		-\sqrt{-1}&1&0\\
		0&0&1
	\end{array}\right),~g_2=\left(\begin{array}{ccc}
		1&0&0\\
		0&1&-\sqrt{-1}\\
		0&-\sqrt{-1}&1
	\end{array}\right)\in\GL_2(\bZ\left[1/2,\sqrt{-1}\right])\]
	\[\begin{array}{cc}
		B_j=g_jB_{\std}g_j^{-1}&(j\in\{1,2\}).
	\end{array}\]

	Write $A=\bZ\left[1/2,\sqrt{-1},x,y,z\right]/(x^2+y^2+z^2-1)$. Then, $\SO(3)$ acts on $\Spec A$ by the restriction of the canonical action of $\SO(3)$ on $\Spec\bZ\left[1/2,\sqrt{-1},x,y,z\right]$. That is, for a $\bZ\left[1/2,\sqrt{-1}\right]$-algebra $R$, each element $g\in\SO(3,R)$ acts on
	$(\Spec A)(R)$
	as a $3\times 3$ matrix via the identification
	$(\Spec A)(R)\cong\left\{(x~y~z)^T\in R^3:x^2+y^2+z^2=1\right\}\subset R^3$.
	\begin{prop}\label{prop:SO(3)/SO(2)=S^2}
		\begin{enumerate}
			\renewcommand{\labelenumi}{(\arabic{enumi})}
			\item We have
			\[\begin{array}{cc}
				B_1\cap \SO(3)=\diag(\SO(2),1),&B_2\cap \SO(3)=\diag(1,\SO(2)).
			\end{array}\]
			\item Let $j\in\{1,2\}$. We have an $\SO(3)$-equivariant isomorphism
			\[\SO(3)/B_j\cap \SO(3)\cong \Spec A.\]
			Define $\SO(3)$-equivariant monomorphisms $i_j$ as
			\[i_j:\Spec A\cong\SO(3)/B_j\cap \SO(3)\hookrightarrow \cB_{\SL_3}.\]
			\item Define an action of $\Gamma$ on $A$ by
			\[\begin{array}{cccc}
				\sigma(\sqrt{-1})=-\sqrt{-1},&\sigma(x)=-x,&\sigma(y)=-y,&\sigma(z)=-z.
			\end{array}\]
			Then, the elements of $\cB_{\SL_3}(A)$ corresponding to $i_1$ and $i_2$ are $\Gamma$-invariant.
		\end{enumerate}
	\end{prop}
	\begin{proof}
		We only prove the assertions for $B_1$. The other is proved in a similar way. For (1), we may prove the equality
		\begin{equation}
			B_{\std}\cap g_1^{-1}\SO(3) g_1=g_1^{-1}\diag(\SO(2),1)g_1.
			\label{eq:B1SO(3)}
		\end{equation}
		by passing to the conjugate by $g_1^{-1}$. Let $R$ be an arbitrary $\bZ\left[1/2,\sqrt{-1}\right]$-algebra. Then, the computation of $\mu_2$ in \cite[Section 3.2]{MR4627704} implies
		\[g_1^{-1}\diag(\SO(2,R),1)g_1=\{\diag(a,a^{-1},1)\in\SL_3(R):~a\in R^\times\}.\]
		
		Let $b=(b_{ij})\in B_{\std}(R)$. Then, $g_1bg_1^{-1}$ belongs to $\SO(3,R)$ if and only if the equality $g_1^{-2}b^Tg_1^{2}b=1$ holds since $g_1$ is symmetric. One can check 
		\[g_1^{-2}b^Tg_1^{2}b\!=\!\left(\begin{array}{ccc}
			b_{11}b_{22}&2b_{12}b_{22}&b_{12}b_{23}\!+\!b_{22}b_{13}\\
			0&b_{33}^{-1}&b_{11}b_{23}\\
			\!-2\sqrt{-1}b_{11}b_{23}
			&\!-2\sqrt{-1}b_{22}b_{13}-2\sqrt{-1}b_{22}b_{13}
			&\!-4\sqrt{-1}b_{23}b_{13}\!+\!b_{33}^2
		\end{array}\right)\]
		(use $b_{11}b_{22}=b^{-1}_{33}$).
		It is now straightforward that $g_1^{-2}b^Tg_1^{2}b=1$ if and only if $b_{22}=b_{11}^{-1}$, $b_{33}=1$, and $b_{12}=b_{23}=b_{13}=0$. This proves the equality \eqref{eq:B1SO(3)}.
		
		We next prove (2). It is easy to show that the stabilizer subgroup of $\SO(3)$ at $(0~0~1)^T\in(\Spec A)(\bZ\left[1/2,\sqrt{-1}\right])$ is $\SO(2)$. We thus obtain a monomorphism
		\[\SO(3)/\SO(2)\hookrightarrow\Spec A.\]
		To see that it is an isomorphism, it will suffice to show that the identity map of $A$ is \'etale locally expressed as $g (0~0~1)^T$ for some $g\in \SO(3)$. Since \[(x^2+y^2)+(y^2+z^2)+(z^2+x^2)=2\in(\bZ\left[1/2,x,y,z\right]/(x^2+y^2+z^2-1))^\times,\]
		the affine schemes $\Spec A\left[1/\sqrt{x^2+y^2}\right]$, $\Spec A\left[1/\sqrt{y^2+z^2}\right]$, and $\Spec A\Big[1/\sqrt{z^2+x^2}\Big]$ form an \'etale cover of $\Spec A$. Set
		\[g_{xy}=\left(\begin{array}{ccc}
			\frac{zx}{\sqrt{x^2+y^2}}&-\frac{y}{\sqrt{x^2+y^2}}&x\\
			\frac{zy}{\sqrt{x^2+y^2}}&\frac{x}{\sqrt{x^2+y^2}}&y\\
			-\sqrt{x^2+y^2}&0&z
		\end{array}\right)\in\SO(3,A\left[1/\sqrt{x^2+y^2}\right]).\]
		Then, we have $g_{xy}(0~0~1)^T=(x~y~z)^T$ on this \'etale locus. One can find similar matrices sending $(0~0~1)^T$ to $(x~y~z)^T$ on the other \'etale loci. This shows (2).
		
		For (3), observe that the automorphism $\sigma$ on $A$ naturally extends to
		$A\left[1/\sqrt{x^2+y^2}\right]$ by $\sigma(\sqrt{x^2+y^2})=\sqrt{x^2+y^2}$. By construction of $i_1$, the Borel subgroups of $\SL_3$ corresponding to $i_1$ on $\Spec A\left[1/\sqrt{x^2+y^2}\right]$ is $g_{xy}B_1g_{xy}^{-1}\in\cB_{\SL_3}(A\left[1/\sqrt{x^2+y^2}\right])$.
		Since $\sigma(g_{xy}g_1)=g_{xy}w_0\sigma(g_1)=g_{xy}g_1w_0$ for $w_0\coloneqq\diag(1,-1,-1)$, $i_1$ is $\Gamma$-invariant on this \'etale locus. Similar arguments work on the other loci. This completes the proof.
	\end{proof}
	
	\begin{rem}
		The argument of (2) clearly works if we replace $A$ by
		\[\bZ\left[1/2,x,y,z\right]/(x^2+y^2+z^2-1).\]
	\end{rem}
	Notice that the structure morphism $\bZ\left[1/2,\sqrt{-1}\right]\to A$ is clearly $\Gamma$-equivariant. Put an action of $\SO(3)$ on $\Spec A^\Gamma\otimes_{\bZ\left[1/2\right]} \bZ\left[1/2,\sqrt{-1}\right]$ by the isomorphism
	\[A\cong A^\Gamma\otimes_{\bZ\left[1/2\right]} \bZ\left[1/2,\sqrt{-1}\right]\]
	(see Theorem \ref{thm:galoisdescent} (2)). In view of Theorem \ref{thm:galoisdescent} (3) and Proposition \ref{prop:descentofaction}, we obtain two $\SO(3)$-equivariant monomorphisms
	$Z_j\coloneqq\Spec A^\Gamma\hookrightarrow \cB_{\SL_3}$,
	which we denote by the same symbol $i_j$. For a digression, we describe $A^\Gamma$:
	\begin{prop}
		Define an action of $\SO(3)$ on
		\[\Spec \bZ\left[1/2,x',y',z'\right]/((x')^2+(y')^2+(z')^2+1)\]
		in a similar way to that on $\Spec A$.
		Then, there is an $\SO(3)$-equivariant isomorphism
		\[\Spec A^\Gamma\cong \Spec \bZ\left[1/2,x',y',z'\right]/((x')^2+(y')^2+(z')^2+1).\]
	\end{prop}
	\begin{proof}
		Define $f:\bZ\left[1/2,x',y',z'\right]/((x')^2+(y')^2+(z')^2+1)\to A^\Gamma$ by
		\[\begin{array}{ccc}
			x'\mapsto\sqrt{-1}x,&y'\mapsto\sqrt{-1}y,&z'\mapsto\sqrt{-1}z.
		\end{array}\]
		In view of Theorem \ref{thm:galoisdescent}, $\bZ\left[1/2,\sqrt{-1}\right]\otimes_{\bZ\left[1/2\right]} f$ can be identified with the map
		\[\bZ\left[1/2,\sqrt{-1},x',y',z'\right]/((x')^2+(y')^2+(z')^2+1)\to A;\]
		\[\begin{array}{cccc}
			\sqrt{-1}\mapsto \sqrt{-1},&x'\mapsto\sqrt{-1}x,&y'\mapsto\sqrt{-1}y,&z'\mapsto\sqrt{-1}z.
		\end{array}\]
		The resulting morphism
		\[\Spec A\to \Spec \bZ\left[1/2,\sqrt{-1},x',y',z'\right]/((x')^2+(y')^2+(z')^2+1)\]
		is clearly an $\SO(3)$-equivariant isomorphism. Since the action of $\SO(3)$ on
		\[\Spec A^\Gamma\otimes_{\bZ\left[1/2\right]}\bZ\left[1/2,\sqrt{-1}\right]\]
		is induced from the action on $\Spec A$ via the isomorphism $A\cong A^\Gamma\otimes_{\bZ\left[1/2\right]}\bZ\left[1/2,\sqrt{-1}\right]$, $\bZ\left[1/2,\sqrt{-1}\right]\otimes_{\bZ\left[1/2\right]}\Spec f$ is an $\SO(3)$-equivariant isomorphism. Since the containment $\bZ\left[1/2\right]\subset\bZ\left[1/2,\sqrt{-1}\right]$ is faithfully flat, $\Spec f$ is an $\SO(3)$-equivariant isomorphism.
	\end{proof}
	
	We demonstrate similar computations to the proof of Theorem \ref{thm:openorbit} (b)$\Rightarrow$(a) to give moduli descriptions of $Z_1$ and $Z_2$, and to prove that $i_1$ and $i_2$ are affine immersions:
	\begin{lem}\label{lem:translation1}
		Let $R$ be a $\bZ\left[1/2\right]$-algebra, and $g\in \SL_3(R)$. Set $B=gB_{\std}g^{-1}$.
		\begin{enumerate}
			\renewcommand{\labelenumi}{(\arabic{enumi})}
			\item The Borel subgroup $B$ belongs to the image of $i_1$ if and only if $c_1(g)=0$ and $c_3(g)\in R^\times$.
			\item The Borel subgroup $B$ belongs to the image of $i_2$ if and only if $c_1(g)\in R^\times$ and $c_2(g)=0$.
		\end{enumerate}
	\end{lem}
	
	\begin{proof}
		We remark that all conditions are local in the \'etale topology by Lemma \ref{lem:imisetalelocal} and Lemma \ref{lem:ffreduction}. Hence we may replace $R$ by $R\left[\sqrt{-1}\right]$ to assume that $R$ is a $\bZ\left[\sqrt{-1}\right]$-algebra. Put
		\[g_1'=\left(\begin{array}{ccc}
			1&-\sqrt{-1}&0\\
			-\sqrt{-1}&1&0\\
			0&0&\frac{1}{2}
		\end{array}\right)\in\SL_3(\bZ\left[1/2,\sqrt{-1}\right]).\]
		One can easily check $B_1=g_1'B_{\std}(g_1')^{-1}$.
		Suppose that $B$ belongs to the image of $i_1$. To prove that $c_1(g)=0$ and $c_3(g)\in R^\times$, we may assume that there exists $k\in\SO(3,R)$ such that $B=kB_1k^{-1}$. Then, \cite[Corollaire 5.3.12 and Proposition 5.1.3]{MR0218363} imply $g^{-1}kg_1'\in B_{\std}(R)$. The assertions $c_1(g)=0$ and $c_3(g)\in R^\times$ now follow from Lemma \ref{lem:kgbreduction}. 
		
		Conversely, suppose that $c_1(g)=0$ and $c_3(g)\in R^\times$. Since $B=gbB_{\std}b^{-1}g^{-1}$ for every element $b\in B_{\std}(R)$, one can replace $g$ by
		\[g\left(\begin{array}{ccc}
			\frac{1}{c_3}&\frac{1-(v_2,v_2)}{2c_3}&0\\
			0&1&0\\
			0&0&c_3
		\end{array}\right)\]
		to assume that $(v_2,v_2)=c_3=1$. We remark that $c_1=0$ still holds from Lemma \ref{lem:kgbreduction}. Set
		\[k=\left(\begin{array}{ccc}
			v_2&-\sqrt{-1}(v_2-v_1)&\sqrt{-1}(v_3+(v_1-v_2,v_3)v_1-(v_1,v_3)v_2)
		\end{array}\right).\]
		By a similar argument to Theorem \ref{thm:openorbit}, $k$ belongs to $\SO(3,R)$. In view of Lemma \ref{lem:kgbreduction}, we may replace $g$ by $k^{-1}g$ to assume $v_1=e_1-\sqrt{-1}e_2$ and $v_2=e_1$. Set
		\[b=\left(\begin{array}{ccc}
			1&\sqrt{-1}&\frac{1}{2}g_{13}\\
			0&-2\sqrt{-1}&-\frac{\sqrt{-1}}{2}g_{13}-\frac{1}{2}g_{23}\\
			0&0&\frac{\sqrt{-1}}{2}
		\end{array}\right).\]
		Then, we have $gb=g_1'$ since $\det g=1$. The assertion now follows from
		\[B=gB_{\std}g^{-1}=g_1'B_{\std}(g_1')^{-1}=B_1.\]
		
		We next prove (2). The ``only if'' direction follows by a similar argument to (1). Suppose that $c_1(g)\in R^\times$ and $c_2(g)=0$. We may pass to the \'etale cover $\Spec R\left[\sqrt{r}_1\right]$, and multiply
		\[\left(\begin{array}{ccc}
			\frac{1}{\sqrt{r}_1}&-\frac{(v_1,v_2)}{c_1}&0\\
			0&1&0\\
			0&0&\sqrt{r}_1
		\end{array}\right)\]
		to $g$ from the right side to assume $c_1=1$ and $(v_1,v_2)=0$. Since $c_2=0$, we have $(v_2,v_2)=0$. Compute $1=(\det g)^2=\det gg^T$ to get $(v_2,v_3)^2=-1$. We then replace $g$ by
		\[g\left(\begin{array}{ccc}
			1&0&0\\
			0&1&\frac{(1-(v_3,v_3))(v_2,v_3)}{2}\\
			0&0&1
		\end{array}\right)\]
		to assume $(v_3,v_3)=1$. It is clear that the equalities
		\[\begin{array}{ccc}
			c_1=1,&(v_1,v_2)=(v_2,v_2)=0,&(v_2,v_3)^2=-1
		\end{array}\]
		still hold. Define a matrix $k$ by
		\[v_1(k)=\sqrt{-1}(v_2(g),v_3(g))v_1(g),\]
		\begin{flalign*}
			v_3(k)
			&=\sqrt{-1}(v_2(g),v_3(g))(v_3(g)-(v_1(g),v_3(g))v_1(g)\\
			&-\frac{1}{2}(v_1(g),v_3(g))^2(v_2(g),v_3(g))v_2(g)),
		\end{flalign*}
		\[v_2(k)=v_2(g)+\sqrt{-1}v_3(k).\]
		Then, $k$ belongs to $\SO(3,R)$. Replace $g$ by $k^{-1}g$ to assume
		\[\begin{array}{cc}
			v_1(g)=\sqrt{-1}(v_2(g),v_3(g))e_1,&v_2(g)=e_2-\sqrt{-1}e_3.
		\end{array}\]
		Set
		\[b=\left(\begin{array}{ccc}
			\frac{\sqrt{-1}(v_2,v_3)}{2}&0&-2g_{13}\\
			0&1&-\sqrt{-1}(1+2g_{23}(v_2,v_3))\\
			0&0&2\sqrt{-1}(v_2,v_3)
		\end{array}\right).\]
		Then, we have
		\[gb=\left(\begin{array}{ccc}
			\frac{1}{2}&0&0\\
			0&1&-\sqrt{-1}\\
			0&-\sqrt{-1}&1
		\end{array}\right)\]
		since $\det g=1$. The assertion now follows from
		\[gB_{\std}g^{-1}=\left(\begin{array}{ccc}
			\frac{1}{2}&0&0\\
			0&1&-\sqrt{-1}\\
			0&-\sqrt{-1}&1
		\end{array}\right)B_{\std}\left(\begin{array}{ccc}
			\frac{1}{2}&0&0\\
			0&1&-\sqrt{-1}\\
			0&-\sqrt{-1}&1
		\end{array}\right)^{-1}=g_2 B_{\std}g^{-1}_2.\]
		This completes the proof.
	\end{proof}

	\begin{property}
		Let $F$ be an algebraically closed field of characteristic different from 2, and $V=(0\subset Fv_1\subset V_2\subset F^3)$ be a flag.
		\begin{enumerate}
			\renewcommand{\labelenumi}{(\arabic{enumi})}
			\item We say $V$ satisfies Property (LC1) if $(v_1,v_1)=0$, and every vector $v\in V_2$ satisfies either $Fv_1=Fv$ or $(v_1,v)\neq 0$.
			\item We say $V$ satisfies Property (LC2) if $(v_1,v_1)\neq 0$, and there exists a nonzero element $v\in V_2$ such that $(v_1,v)=(v,v)=0$. We remark that if such $v$ exists, $v$ is uniquely determined up to nonzero scalar by the property $(v_1,v)=0$ since $\dim V_2=2$.
		\end{enumerate}
	\end{property}
	\begin{lem}\label{lem:flagcriterion1}
		Let $R$ be a $\bZ\left[1/2\right]$-algebra, and $g\in \SL_3(R)$.
		\begin{enumerate}
			\renewcommand{\labelenumi}{(\arabic{enumi})}
			\item If $c_1(g)=0$, and $c_3(g)\in R^\times$ then the flags corresponding to all geometric fibers of $B$ satisfy Property (LC1). The converse holds if $R$ is reduced.
			\item If $c_1(g)\in R^\times$, and $c_2(g)=0$ then the flags corresponding to all geometric fibers of $B$ satisfy Property (LC2). The converse holds if $R$ is reduced.
		\end{enumerate}
	\end{lem}
	\begin{proof}
		We remark that if $R$ is reduced, an element $r\in R$ is zero if and only if it is zero in all residue fields of $R$. Therefore, we may assume that $R=F$ is an algebraically closed field. 
		
		Part (1) follows since the flag $V=(0\subset Fv_1\subset Fv_1\oplus Fv_2\subset F^3)$ satisfies (LC1) if and only if $(v_1,v_1)=0$ and $(v_1,v_2)\neq 0$. We next prove that $V$ satisfies (LC2) if and only if $c_1(g)\in F^\times$ and $c_2(g)=0$. Notice that $c_1(g)\neq 0$ holds under the both conditions. We also have
		\[\begin{array}{ccc}
			v_2-\frac{(v_1,v_2)}{c_1}v_1\neq 0,
			&(v_2-\frac{(v_1,v_2)}{c_1}v_1,v_1)!=\!0,
			&c_2!=\!c_1(v_2-\frac{(v_1,v_2)}{c_1}v_1,v_2-\frac{(v_1,v_2)}{c_1}v_1).
		\end{array}\]
		The equivalence of (2) is now obvious.
	\end{proof}
	Let $\hat{Z}_1$ (resp.\ $\hat{Z}_2$) be the moduli space of flags with Property (LC1) (resp.\ (LC2)). That is, for a $\bZ\left[1/2\right]$-algebra $R$, $\hat{Z}_1(R)$ (resp.\ $\hat{Z}_2(R)$) consists of Borel subgroups of $\SL_3$ over $R$ whose corresponding flags at geometric points of $\Spec R$ satisfy (LC1) (resp.\ (LC2)).
	\begin{thm}\label{thm:modulimiddles}
		Let $R$ be a reduced $\bZ\left[1/2\right]$-algebra, and $B\in\cB_{\SL_3}(R)$.
		\begin{enumerate}
			\renewcommand{\labelenumi}{(\arabic{enumi})}
			\item The following conditions are equivalent:
			\begin{enumerate}
				\item[(a)] $B$ belongs to the image of $i_1$.
				\item[(b)] The flags corresponding to all geometric fibers of $B$ satisfy Property (LC1).
			\end{enumerate}
			\item The following conditions are equivalent:
			\begin{enumerate}
				\item[(a)] $B$ belongs to the image of $i_2$.
				\item[(b)] The flags corresponding to all geometric fibers of $B$ satisfy Property (LC2).
			\end{enumerate}
		\end{enumerate}
		In particular, $Z_1$ (resp.\ $Z_2$) is the left Kan extension of the copresheaf $\hat{Z}_1|_{\CAlg_{\bZ\left[1/2\right]}^{\red}}$ (resp.\ $\hat{Z}_2|_{\CAlg_{\bZ\left[1/2\right]}^{\red}}$).
	\end{thm}
	\begin{proof}
		Since all conditions are local in the \'etale topology, we may assume that $B$ is $\SL_3$-conjugate to $B_{\std}$. Then, the equivalences follow from Lemma \ref{lem:translation1} and Lemma \ref{lem:flagcriterion1}.
	\end{proof}
	\begin{thm}\label{thm:iisimmersion}
		The morphisms $i_1$ and $i_2$ are affine immersions.
	\end{thm}
	
	\begin{proof}
		We only prove (1). It will suffice to prove $i_1:\Spec A\to\cB_{\SL_3}$ is an affine immersion by taking the base change $-\otimes_{\bZ\left[1/2\right]}\bZ\left[1/2,\sqrt{-1}\right]$. Let $R$ be a $\bZ\left[1/2,\sqrt{-1}\right]$-algebra, and $g\in\SL_3(R)$. Consider the morphism $\Spec R\to \cB_{\SL_3}$ determined by $B\coloneqq g B_{\std}g^{-1}$. By a similar argument to Theorem \ref{thm:openorbit} (2), it will suffice to show that the projection $\Spec A\times_{\cB_{\SL_3}}\Spec R\to \Spec R$ is represented by the affine immersion $\Spec R_{c_1(g)}/(c_2(g))\hookrightarrow\Spec R$. For this, it will suffice to show that for a $\bZ\left[1/2,\sqrt{-1}\right]$-algebra $S$, a $\bZ\left[1/2,\sqrt{-1}\right]$-algebra homomorphism $f\in (\Spec R)(S)$ belongs to $(\Spec A\times_{\cB_{\SL_3}}\Spec R)(S)\subset(\Spec R)(S)$ if and only if $f(c_1)\in S^\times$ and $f(c_2)=0$. This follows from Lemma \ref{lem:translation1}.
	\end{proof}
	
	\subsection{Closed Subscheme $\cB_{\SO(3)}$}\label{sec:closedorbit}
	
	Define an involution $\theta$ on $\SL_3$ by $\theta(g)=(g^T)^{-1}$. Let $\cB^\theta_{\SL_3}$ be the moduli scheme of $\theta$-stable Borel subgroups of $\SL_3$ (see \cite[Lemma 3.1.1]{MR4627704}). A similar argument to \cite[Proposition 5.5.2]{hayashijanuszewski} implies that the map $B\mapsto B\cap \SO(3)$ determines an $\SO(3)$-equivariant isomorphism $\cB_{\SL_3}^\theta\cong\cB_{\SO(3)}$. Write $i_{\clo}$ for the composite arrow $\cB_{\SO(3)}\cong \cB_{\SL_3}^\theta\hookrightarrow\cB_{\SL_3}$.
	
	\begin{lem}\label{lem:iclo}
		The morphism $i_{\clo}$ is a closed immersion.
	\end{lem}
	
	\begin{proof}
		This follows since $\cB^\theta_{\SL_3}$ is a closed subscheme of $\cB_{\SL_3}$; see \cite[Lemma 3.1.1]{MR4627704}.
	\end{proof}
	
	As in the former sections, let us study $R$-points of $\cB^\theta_{\SL_3}$ in terms of flags at fibers.
	
	\begin{lem}\label{lem:translation2}
		Let $R$ be a $\bZ\left[1/2\right]$-algebra, and $g\in \SL_3(R)$. Then, the Borel subgroup $B\coloneqq gB_{\std}g^{-1}$ is $\theta$-stable if and only if $c_1(g)=c_3(g)=0$.
	\end{lem}
	\begin{proof}
		Recall that $B^T_{\std}=(-K_3)B_{\std}(-K_3)^{-1}$. Hence $gB_{\std}g^{-1}$ is $\theta$-stable if and only if $-K_3g^Tg$ is upper triangular. It follows by the direct computation of $K_3g^Tg$ that it is equivalent to the condition that $c_1(g)=c_3(g)=0$.
	\end{proof}
	\begin{property}
		Let $F$ be an algebraically closed field of characteristic diffrent from 2. We say that a flag $V=(0\subset Fv_1\subset V_2\subset F^3)$ satisfies Property (C) if $(v_1,v_1)=0$, and every vector $v\in V_2$ satisfies $(v_1,v)=0$.
	\end{property}
	\begin{lem}\label{lem:flagcriterion2}
		Let $R$ be a $\bZ\left[1/2\right]$-algebra, and $g\in \SL_3(R)$. If $c_1(g)=c_3(g)=0$ then the flags corresponding to all geometric fibers of $B$ satisfy Property (C). The converse holds if $R$ is reduced.
	\end{lem}
	\begin{proof}
		By a similar argument to Lemma \ref{lem:flagcriterion1}, we nay assume that $R=F$ is an algebraically closed of characteristic $\neq 2$. Then, the flag corresponding to $gB_{\std}g^{-1}$ is given by $(0\subset Fv_1\subset Fv_1\oplus Fv_2\subset F^3)$. Under the both conditions, we clearly have $c_1(g)=0$. We may, therefore, assume $c_1(g)=0$. Then, it is clear that $(v_1,v_2)=0$ if and only if every vector $v\in Fv_1\oplus Fv_2$ satisfies $(v_1,v)=0$.
	\end{proof}
	\begin{thm}\label{thm:moduliclosed}
		Let $R$ be a reduced $\bZ\left[1/2\right]$-algebra. For a Borel subgroup $B\in\cB_{\SL_3}(R)$, the following conditions are equivalent:
		\begin{enumerate}
			\item[(a)] $B$ belongs to the image of $i_{\clo}$.
			\item[(b)] The flags corresponding to all geometric fibers of $B$ satisfy Property (C).
		\end{enumerate}
	\end{thm}
	\begin{proof}
		This is proved in a similar way to Theorem \ref{thm:modulimiddles}.
	\end{proof}
	\section{$\SO(3)$-Homogeneous Decomposition of $\cB_{\SL_3}$ over $\bZ\left[1/2\right]$}\label{sec:kgb}
	\begin{thm}\label{mainthm}
		\begin{enumerate}
			\renewcommand{\labelenumi}{(\arabic{enumi})}
			\item The set
			\[\{i_{\op}:U\hookrightarrow \cB_{\SL_3},i_1:Z_1\hookrightarrow \cB_{\SL_3},i_2:Z_2\hookrightarrow \cB_{\SL_3},i_{\clo}:\cB_{\SO(3)}\hookrightarrow \cB_{\SL_3}\}\]
			exhibits a set-theoretic decomposition of $\cB_{\SL_3}$.
			\item The subschemes $Z_1$, $Z_2$, and $\cB_{\SO(3)}$ are not $\SO(3)$-orbits, but $\SO(3)$-homogeneous in the \'etale topology in the sense of \cite[Proposition et d\'efintion 6.7.3]{MR0237513}.
		\end{enumerate}
	\end{thm}
	
	\begin{proof}
		For (1), let $F$ be an algebraically closed field of characteristic not equal to 2. Identify $\cB_{\SL_3}(F)$ with the set of full flags of $F^3$. Notice that every flag of $F^3$ has only one of Property (O), (LC1), (LC2), and (C). The assertion now follows from Theorem \ref{thm:openorbit}, Theorem \ref{thm:modulimiddles}, Theorem \ref{thm:moduliclosed}, and Theorem \ref{thm:fieldwisedec}.
		
		The first part of (2) follows since these three subschemes do not admit $\bZ\left[1/2\right]$-points. The latter assertion of (2) is evident by the constructions of these three subschemes.
	\end{proof}

	\appendix
	\section{Descent Technique}\label{sec:descent}
	In this section, we develop the descent results we use in this paper. The first result is easy:
	
	\begin{lem}\label{lem:ffreduction}
		Let $f:R\to S$ be a faithfully flat homomorphism of commutative rings.
		\begin{enumerate}
			\renewcommand{\labelenumi}{(\arabic{enumi})}
			\item An element $r\in R$ is zero if and only if $f(r)=0$.
			\item An element $r\in R$ belongs to $R^\times$ if and only if $f(r)\in S^\times$.
		\end{enumerate}
	\end{lem}
	
	\begin{proof}
		The ``only if'' direction in each assertion is obvious. The ``if'' direction of (1) follows since $f$ is injective. In fact, we have a canonical equalizer sequence
		\[R\overset{f}{\to} S\rightrightarrows S\otimes_R S.\]
		
		We next prove the ``if'' direction of (2). Suppose that we are given an element $r\in R$ such that $f(r)\in S^\times$. To see $r\in R^\times$, it will suffice to show $r\not\in \fp$ for every prime ideal $\fp$ of $R$. Since $f$ is faithfully flat, there exists a prime ideal $\fq\subset S$ such that $f^{-1}(\fq)=\fp$. Since $f(r)$ is a unit of $S$, we have $f(r)\not\in \fq$. We thus obtain $r\not\in f^{-1}(\fq)=\fp$ as desired. This completes the proof.
	\end{proof}
	
	In this paper, we try to find a smaller ring of definition of schemes and morphisms. Firstly, let us give some remarks on the Galois descent of morphisms.
	
	\begin{lem}\label{lem:meaningofdescent}
		Let $k\to k'$ be a faithfully flat homomorphism of commutative rings, $X$ be a $k$-space, and $Y$ be a $k$-sheaf in the fpqc topology. Then, for a morphism $f':X\otimes_k k'\to Y\otimes_k k'$ of $k'$-schemes, there is at most unique morphism $f:X\to Y$ such that $f'=f\otimes_k k'$.
	\end{lem}
	
	We remark that if $k\to k'$ is \'etale, $Y$ can be an \'etale $k$-sheaf.
	
	\begin{proof}
		Let $g$ be another morphism satisfying $f'=g\otimes_k k'$. For each commutative $k$-algebra $R$, the canonical homomorphism $R\to R\otimes_k k'$ is faithfully flat. Since
		\[Y(R)\to Y(R\otimes_k k')\rightrightarrows Y((R\otimes_k k')\otimes_{R} (R\otimes_k k'))\]
		is an equalizer sequence, $Y(R)\to Y(R\otimes_k k')$ is injective. The uniqueness now follows from the commutative diagram
		\[\begin{tikzcd}
			X(R)\arrow[r]\arrow[d,shift left=0.7ex, "g_T"]
			\arrow[d,shift right=0.7ex, "f_R"']
			&X(R\otimes_k k')\arrow[d, "f_{R\otimes_k k'}=g_{R\otimes_k k'}=f'_{R\otimes_k k'}"]\\
			Y(R)\arrow[r, hook]&Y(R\otimes_k k').
		\end{tikzcd}\]
	\end{proof}
	
	The descent of $f'$ is to find the unique morphism $f$ in the above lemma. The next result enables us to descend rings of definition of morphisms and their domains simultaneously.
	
	\begin{thm}\label{thm:galoisdescent}
		Let $i:k\to k'$ be a Galois extension of commutative rings of Galois group $\Gamma$. Let $S$ be a commutative ring, equipped with an action of $\Gamma$ and a $\Gamma$-equivariant ring homomorphism $g:k'\to S$. We denote $R=S^\Gamma$, and $j:R\to S$ be the canonical embedding. In particular, $g\circ i:k\to S$ is factorized into $j\circ f$ for a unique homomorphism $f:k\to R$. Let $X$ be an \'etale $k$-sheaf.
		\begin{enumerate}
			\renewcommand{\labelenumi}{(\arabic{enumi})}
			\item The homomorphism $j:R\to S$ is a Galois extension of Galois group $\Gamma$.
			\item Put an action of $\Gamma$ on $R\otimes_k k'$ by the base change. Then, we have a canonical $\Gamma$-equivariant isomorphism $R\otimes_k k'\cong S$ which we denote by $(j,g)$. In particular, the canonical homomorphism $R\to R\otimes_k k'$ is a Galois extension of Galois group $\Gamma$ for this action.
			\item Let $\alpha\in X(S)^\Gamma\subset X(S)=(X\otimes_k k')(S)$. Let $\bar{\alpha}\in X(R)$ be an element satisfying $X(j)(\bar{\alpha})=\alpha$ which uniquely exists by Galois descent. We denote the corresponding morphisms $\alpha:\Spec S\to X\otimes_k k'$ and $\Spec R\to X$ by the same symbols $\alpha$ and $\bar{\alpha}$ respectively. Then, the composite map
			\[\Spec R\otimes_k k'\overset{(j,g)^{-1}}{\cong} \Spec S\overset{\alpha}{\to} X\otimes_k k'\]
			coincides with the base change of $\bar{\alpha}$.
		\end{enumerate}
	\end{thm}
	\begin{proof}
		Part (1) and (2) are proved in a similar way to \cite[Theorems 14.86 and 14.85]{MR4225278}. 
		
		For (3), it will suffice to compare the images of $\id_{R\otimes_k k'}\in (\Spec R\otimes_k k')(R\otimes_k k')$ by the Yoneda lemma. The image along the map constructed in (3) is $X((j,g))^{-1}(\alpha)\in X(R\otimes_k k')=(X\otimes_k k')(R\otimes_k k')$. Let $l:R\to R\otimes_k k'$ denote the canonical map. Then, the image along $\bar{\alpha}\otimes_k k'$ is computed as
		\[X(l)(\bar{\alpha})
		=X((j,g))^{-1}X(j)(\bar{\alpha})
		=X((j,g))^{-1}(\alpha).\]
		Hence the two images coincide. This completes the proof.
	\end{proof}
	
	For the relation of descent of spaces and actions of groups, the following result is useful:
	
	\begin{prop}\label{prop:descentofaction}
		Let $k\to k'$ be a faithfully flat ane \'etale homomorphism of commutative rings. Let $i:X\to Y$ be a monomorphism of \'etale $k$-sheaves, and $K$ be a group \'etale $k$-sheaf. Suppose that $K$ acts on $G$. If the induced action of $K\otimes_k k'$ on $Y\otimes_k k'$ restricts to $X\otimes_k k'$, the action of $K$ on $Y$ restricts to $X$. Moreover, the base change of the resulting action on $X$ coincides with the given action on $X\otimes_k k'$.
	\end{prop}
	\begin{proof}
		We denote the action map $K\times Y\to Y$ (resp.\ $(K\otimes_k k')\times (X\otimes_k k')\to X\otimes_k k'$) by $\psi$ (resp.\ $\phi'$). Let $R$ be a $k$-algebra. Write $l:R\to R\otimes_k k'$ for the canonical homomorphism. For $j\in\{1,2\}$, let $\iota_j:R\otimes_k k'\to (R\otimes_k k')\otimes_R (R\otimes_k k')$ be the canonical homomorphism onto the $j$th factor. Take $g\in K(R)$ and $x\in X(R)$. We check the descent condition for $\phi'\circ(K\times X)(l)(g,x)$ along $l$. Observe that for $j\in\{1,2\}$, we have
		\[\begin{split}
			i\circ X(\iota_j)\circ\phi'\circ(K\times X)(l)(g,x)&=Y(\iota_j)\circ i\circ\phi'\circ(K\times X)(l)(g,x)\\
			&=Y(\iota_j)\circ \psi\circ (K\times i)\circ (K\times X)(l)(g,x)\\
			&=\psi\circ(K\times Y)(\iota_j)\circ (K\times i)\circ(K\times X)(l)(g,x)\\
			&=\psi\circ(K\times i)\circ (K\times X)(\iota_j)\circ(K\times X)(l)(g,x)\\
			&=\psi\circ(K\times i)\circ (K\times X)(\iota_j\circ l)(g,x).
		\end{split}\]
		We used the hypothesis that $i\otimes_k k'$ is $K\otimes_k k'$-equivariant for the second equality. Since $\iota_1\circ l=\iota_2\circ l$, we have
		\[\begin{split}
			i\circ X(\iota_1)\circ\phi'\circ(K\times X)(l)(g,x)&=\psi\circ(K\times i)\circ (K\times X)(\iota_1\circ l)(g,x)\\
			&=\psi\circ(K\times i)\circ (K\times X)(\iota_2\circ l)(g,x)\\
			&=i\circ X(\iota_2)\circ\phi'\circ(K\times X)(l)(g,x).
		\end{split}\]
		Since $i$ is monic, we have $X(\iota_1)\circ\phi'\circ(K\times X)(l)(g,x)=X(\iota_2)\circ\phi'\circ(K\times X)(l)(g,x)$. Since $X$ is an \'etale sheaf, there is a unique element $\phi(g,x)\in X(R)$ such that $X(l)(\phi(g,x))=\phi'\circ(K\times X)(l)(g,x)$. To see that this gives the restriction of the action of $K$ on $Y$ to $X$, notice that
		\[\begin{split}
			Y(l)\circ i(\phi(g,x))
			&=i\circ X(l)(\phi(g,x))\\
			&=i\circ\phi'\circ(K\times X)(l)(g,x)\\
			&=\psi\circ (K\times i)\circ(K\times X)(l)(g,x)\\
			&=\psi\circ (K\times Y)(l)\circ(K\times i)(g,x)\\
			&=Y(l)\circ\psi\circ (K\times i)(g,x)\\
			&=Y(l)\circ\psi(g,i(x)).
		\end{split}\]
		Since $l$ is \'etale and faithfully flat, $Y(l)$ is injective. Therefore, we get the equality $i(\phi(g,x))=\psi(g,i(x))$ as desired. In particular, $\phi$ determines the restriction of the action of $\psi$ (use the hypothesis that $i$ is monic for the naturality of $\phi$).
		
		The proof is completed by showing $\phi'=\phi\otimes_k k'$. Let $R'$ be a $k'$-algebra, $g\in(K\otimes_k k')(R')=K(R')$, and $x\in(X\otimes_k k')(R')=X(R')$. Then, we have
		\[i(\phi'(g,x))=\psi\circ(K\times i)(g,x)
		=\psi(g,i(x))
		=i(\phi(g,x)),\]
		Since $i$ is monic, we have $\phi'(g,x)=\phi(g,x)$.
	\end{proof}
	In Section \ref{sec:middleorbit}, we consider orbit sheaves in the scheme $\cB_{\SL_3}$. To compute their moduli description, the following observation is useful since many objects are local in the \'etale topology:
	\begin{lem}\label{lem:imisetalelocal}
		Let $k$ be a commutative ring, and $f:R\to R'$ be a faithfully flat \'etale homomorphism of $k$-algebras. Let $i:\cF\to \cG$ be a monomorphism of \'etale $k$-sheaves, and $y\in\cG(R)$. Then, $y$ belongs to the image of $i$ if and only if $\cG(f)(y)$ is so.
	\end{lem}
	\begin{proof}
		The ``only if'' direction is clear. To see the ``if'' direction, suppose that we are given an element $x'\in\cF(R')$ such that $i(x')=\cG(f)(y)$. For $j\in\{1,2\}$, let $\iota_j$ denote the caonical homomorphism $R'\to R'\otimes_R R'$ onto the $j$th factor. Then, we have
		\[\begin{array}{cccc}
			i(\cF(\iota_1)(x'))&=\cG(\iota_1)(i(x'))
			&=\cG(\iota_1\circ f)(y)
			&=\cG(\iota_2\circ f)(y)\\
			&=\cG(\iota_2)(i(x'))
			&=i(\cF(\iota_2)(x')).
		\end{array}\]
		This implies $\cF(\iota_1)(x')=\cF(\iota_2)(x')$ since $i$ a monomorphism. Since $\cF$ is an \'etale $k$-sheaf, there is a (unique) element $x\in\cF(R)$ such that $\cF(f)(x)=x'$. The proof is completed by showing $i(x)=y$. This follows from the equality
		\[\cG(f)(i(x))=i(\cF(f)(x))=i(x')=\cG(f)(y)\]
		since $\cG$ is an \'etale $k$-sheaf.
	\end{proof}
	\section{$\cB_{\SO(3)}$ as a Subspace of $\bP^2$}\label{sec:flagschofso(3)}
	In \cite{MR0218363}, the moduli space $\mathcal B_{\SO(3)}$ of Borel subgroups of $\SO(3)$ is proved to be represented by a projective scheme over $\bZ\left[1/2\right]$. The key result for its proof is that the moduli space of subgroups of $\SO(3)$ of type (R) is represented by a quasi-projective scheme. In this appendix, we realize $\mathcal B_{\SO(3)}$ as a moduli subspace of $\bP^2=\Proj\bZ\left[1/2,x,y,z\right]$ by using ideas in this paper. To be precise, set $Z'=\Proj\bZ\left[1/2,x,y,z\right]/(x^2+y^2+z^2)\subset\bP^2$. We establish an $\SO(3)$-equivariant isomorphism $\cB_{\SO(3)}\cong Z'$. We also show that $|\bP^2|=|Z'|\coprod |U'|$ is a $\bZ\left[1/2\right]$-form of the $\SO(3)$-orbit decomposition of $\bP^2$ over $\bZ\left[1/2,\sqrt{-1}\right]$, where $U'=\bP^2\setminus Z'$.
	
	To achieve them, let us recall a moduli description of $\bP^2$: for a commutative $\bZ\left[1/2\right]$-algebra $R$, $\bP^2(R)$ is naturally bijective to the set of equivalence classes of line bundles $\cL$ on $\Spec R$ with generators $(a_1,a_2,a_3)$. If $\cL$ is the structure sheaf $\cO_{\Spec R}$ of $\Spec R$, we will denote it by
	$[a_1~a_2~a_3]^T$.
	
	\begin{property}
		Let $F$ be an algebraically closed field of characteristic different from 2. We say that a one dimensional subspace $V=Fv\subset F^3$ satisfies Property (O)' (resp.\ (C)') if $(v,v)\neq 0$ (resp.\ $(v,v)= 0$).
	\end{property}
	Notice that for a field $F$ (of characteristic $\neq 2$), $F$-points of $\bP^2$ are identified with one dimensional subspaces of $F^3$ by the correspondence
	\[[a_1~a_2~a_3]^T\leftrightarrow F(a_1~a_2~a_3)^T.\]
	The open subscheme $U'$ is the moduli space of one dimensional subspaces $V\subset F^3$ satisfying Property (O)'. That is, for a $\bZ\left[1/2\right]$-algebra $R$,
	\[U'(R)=\{\cV\in\bP^2(R):\ {\rm the\ geometric\ fibers\ of\ }\cV{\rm \ satisfy\ (O)'}\}.\]
	Since the condition (O)' is stable under the action of $\SO(3)$, $U'$ is an $\SO(3)$-invariant open subscheme of $\bP^2$.
	
	\begin{thm}
		There is an $\SO(3)$-equivariant isomorphism
		$\SO(3)/\SO(2)\cong U'$.
	\end{thm}
	
	\begin{proof}
		Take the $\SO(3)$-orbit attached to $[0~0~1]^T$ to get a monomorphism
		\[\SO(3)/\SO(2)\hookrightarrow U'.\]
		
		To see that this morphism is epic, take an arbitrary $\bZ\left[1/2\right]$-algebra $R$ and an $R$-point $\cV\in U'(R)$. We wish to show that $\cV$ lies in the image of the above morphism. We may identify $\cV$ with a pair of a line bundle $\cL$ on $\Spec R$ and its global sections $(a_1,a_2,a_3)$ of $\mathcal L$ which (locally) generate $\mathcal L$ by \cite[Theorem 7.1]{MR0463157}.
		
		Since the assertion is Zariski local by Lemma \ref{lem:imisetalelocal}, we may assume $\mathcal L$ is the coordinate ring of $\Spec R$. In particular, $a_1,a_2,a_3$ are generators of $R$ as an $R$-module with $a^2_1+a^2_2+a^2_3\in R^\times$. We may replace $R$ by $R\left[\sqrt{a^2_1+a^2_2+a^2_3}\right]$ to assume $a^2_1+a^2_2+a^2_3=1$. Then, use the matrices $g_{xy}, g_{yz}, g_{zx}$ in the proof of Proposition \ref{prop:SO(3)/SO(2)=S^2} to see that there \'etale locally exists $g'\in \SO(3,R)$ such that $g'[0~0~1]^T=[a_1~a_2~a_3]^T$. This completes the proof.
	\end{proof}
	We next study $Z'$. Let $P$ be the parabolic subgroup of $\SL_3$ defined by
	\[P(A)=\{g=(g_{ij})\in\SL_3(A):~g_{21}=g_{31}=0\}.\]
	Take the $\SL_3$-orbit of $\bP^2$ attached to $[1~0~0]^T$ to get an isomorphism $\SL_3/P\cong\bP^2$ (see \cite[I.5.6 (3)]{MR2015057}).
	
	We can define an $\SO(3)$-equivariant morphism
	\[i':\cB_{\SO(3)}\overset{i_{\clo}}{\hookrightarrow}
	\cB_{\SL_3}\cong \SL_3/B_{\std}\to \SL_3/P\cong\bP^2,\]
	where the map $\SL_3/B_{\std}\to\SL_3/P$ is the quotient map attached to $B_{\std}\subset P$.
	\begin{prop}\label{prop:i'iscloimm}
		The morphism $i'$ is an $\SO(3)$-equivariant closed immersion.
	\end{prop}
	To prove this, we realize $i'$ as composition of morphisms between quotient spaces over $\bZ\left[1/2,\sqrt{-1}\right]$.
	
	Set
	\[g_{\clo}=\left(\begin{array}{ccc}
		1&0&-\sqrt{-1}\\
		-\sqrt{-1}&0&1\\
		0&1&0
	\end{array}\right).\]
	Define a cocharacter $\mu:\bG_m\to \SO(3)$ over $\bZ\left[1/2,\sqrt{-1}\right]$ by
	\[\mu(a)=g_{\clo}\diag(a,a^{-1},1)g_{\clo}^{-1}.\]
	Define a $\theta$-stable Borel subgroup of $\SL_3$ over $\bZ\left[1/2\right]$ by $B_{\clo}=P_{\SL_3}(\mu)$ ($\theta=((-)^T)^{-1}$). It is evident by definitions that $B_{\clo}=g_{\clo} B_{\std} g_{\clo}^{-1}$. Set
	\[B_{\clo,\SO(3)}=\SO(3)\cap B_{\clo}.\]
	\begin{lem}\label{lem:intersectionso3}
		We have $B_{\clo,\SO(3)}=\SO(3)\cap g_{\clo}Pg_{\clo}^{-1}$.
	\end{lem}
	\begin{proof}
		Recall that $B_{\clo,\SO(3)}$ is a Borel subgroup of $\SO(3)$ (\cite[Proposition 5.2.5 (2)]{hayashijanuszewski}). In particular, $B_{\clo,\SO(3)}$ is flat over $\bZ\left[1/2,\sqrt{-1}\right]$. In view of \cite[Corollaire 17.9.5]{MR0238860}, we may pass to geometric fibers.
		
		Let $\fso(3)$ and $\fp$ be the Lie algebras of $\SO(3)$ and $P$ respectively. Then, we have
		\[2=\dim B_{\clo,\SO(3)}\leq\dim\SO(3)\cap g_{\clo}Pg_{\clo}^{-1}\leq \dim\fso(3)\cap g_{\clo}\fp g_{\clo}^{-1}=2,\]
		where the last equality is followed by straightforward computations. In particular, $\SO(3)\cap g_{\clo}Pg_{\clo}^{-1}$ is smooth of dimension $2$ (\cite[Proposition 1.37]{MR3729270}). Since $B_{\clo,\SO(3)}$ is a Borel subgroup of $\SO(3)$,
		we have $B_{\clo,\SO(3)}=\SO(3)\cap g_{\clo} P g_{\clo}^{-1}$.
	\end{proof}
	
	\begin{proof}[Proof of Proposition \ref{prop:i'iscloimm}]
		It will suffice to prove that $i'$ is a monomorphism. We may work over $\bZ\left[1/2,\sqrt{-1}\right]$. Then, $i'$ can be identified with
		\[\SO(3)/B_{\clo,\SO(3)}
		\hookrightarrow
		\SL_3/B_{\clo}\cong\SL_3/B_{\std}\to\SL_3/P,\]
		where $\SL_3/B_{\clo}\cong\SL_3/B_{\std}$ is defined by $gB_{\clo}\mapsto gg_{\clo}B_{\std}$. The assertion now follows from Lemma \ref{lem:intersectionso3}.
	\end{proof}
	
	To relate $\cB_{\SO(3)}$ with $Z'$, let us recall the moduli description of $Z'$.
	
	\begin{property}
		Let $R$ be a commutative $\bZ\left[1/2\right]$-algebra. We say an $R$-point $(\cL,(a_1,a_2,a_3))\in \bP^2(R)$ satisfies Property (C)' if
		$\sum_{i=1}^3 a_i\otimes a_i=0$
		as a global section of $\cL\otimes_{\cO_{\Spec R}}\cL$.
	\end{property}
	
	Then, for each $R$, $Z'(R)$ consists of $R$-points of $\bP^2(R)$ satisfying Property (C)'. If $R$ is reduced, a point $(\cL,(a_1,a_2,a_3))\in \bP^2(R)$ satisfies (C)' if and only if its geometric fibers satisfy (C)'.
	
	\begin{thm}
		The map $i'$ is an isomorphism onto $Z'$.
	\end{thm}
	\begin{proof}
		Observe that the structure morphisms
		\[\begin{array}{cc}
			\cB_{\SO(3)}\to\Spec\bZ\left[1/2\right],&Z'\to\Spec\bZ\left[1/2\right]
		\end{array}\]
		are smooth surjective (see \cite[4 Example 3.37]{MR1917232} for the smoothness of $Z'$). Since the reducedness is local in the smooth topology, $\cB_{\SO(3)}$ and $Z'$ are reduced. In view of the uniqueness of reduced structure on the underlying set of a closed subscheme, it will suffice to show that $|i'|$ is surjective onto $Z'$.
		
		Let $F$ be an algebraically closed field over $\bZ\left[1/2,\sqrt{-1}\right]$. Notice that
		\[i'(B_{\clo,\SO(3)})=F(1~-\sqrt{-1}~0)^T \in Z'(F).\]
		Since $\SO(3,F)$ acts transitively on $\cB_{\SO(3)}(F)$, $i'_F$ factors through $Z'(F)$. Let $Fv\in Z'(F)$. Then, we can choose a vector $u\in F^3$ such that $(v,u)\neq 0$. By the proof of Lemma \ref{lem:translation1} (1), there is a matrix $k\in\SO(3,F)$ such that
		$F k(1~-\sqrt{-1}~0)^T=Fv$.
		This completes the proof.
	\end{proof}
	
	\begin{cor}\label{cor:kgp}
		The $\SO(3)$-equivariant immersions
		\[\begin{array}{cc}
			\SO(3)/\SO(2)\cong U'\subset \bP^2\cong \SL_3/P,
			&\cB_{\SO(3)}\overset{i_{\clo}}{\hookrightarrow}
			\cB_{\SL_3}\cong \SL_3/B_{\std}\to \SL_3/P	
		\end{array}\]
		form a set-theoretic decomposition of $\SL_3/P$.
	\end{cor}
	
	One can think of Corollary \ref{cor:kgp} as a $\bZ\left[1/2\right]$-analog of the $\SO(3,\bC)$-orbit decomposition of $\SL_3(\bC)/P(\bC)$.
	\renewcommand{\abstractname}{Acknowledgements}
	\begin{abstract}
		I am grateful to the referees for careful reading of this manuscript. I would like to thank one of them for pointing out Remark \ref{rem:referee}.
	\end{abstract}
	
	\section*{Funding}
	This work was supported by JSPS KAKENHI Grant Number 21J00023.
	
	\section*{Declarations}
	\subsection*{Conflict of Interest}
	The author declares no competing interests.

\end{document}